\documentclass[11pt,reqno]{amsart}
\usepackage{amssymb}
\usepackage{amssymb,amsmath,amsthm,amsfonts,mathrsfs,hyperref}
\usepackage{times}
\usepackage{enumerate}
\usepackage{cite,titletoc}
\usepackage{indentfirst}
\allowdisplaybreaks

\hypersetup{colorlinks=true,linkcolor=blue,citecolor=red,urlcolor=cyan}

\theoremstyle{plain}
\newtheorem{theorem}{Theorem}[section]

\newtheorem{lemma}[theorem]{Lemma}

\theoremstyle{definition}

\numberwithin{equation}{section}

\title[Existence of positive solutions to anisotropic elliptic equations]{Existence of positive solutions to anisotropic elliptic equations in the entire Euclidean space }

 \author[Rulong Xie, Hailian Wang and Lei Yang]{}

\subjclass{Primary: 35J60, 35A15; Secondary: 35B33}
\keywords{$N$-anisotropic Laplacian, anisotropic singular Moser-Trudinger inequality, exponential growth, mountain-pass theorem, Ekeland's variational principle.}

\thanks{*Corresponding author: Hailian Wang (hailianmath@163.com), Rulong Xie (xierl@ustc.edu.cn), Lei Yang (yanglei21@ahnu.edu.cn)}

\begin{document}
\maketitle

\centerline{\scshape
Rulong Xie$^{1,2}$, Hailian Wang$^{1,*}$
and Lei Yang$^{1,2}$}

\medskip

{\footnotesize
\centerline{$^1$School of Mathematics and Big Data, Chaohu University, Hefei, 238024, China}
}
\medskip

{\footnotesize
\centerline{$^2$School of Mathematics and Statistics, Anhui Normal University, Wuhu, 241002, China}
}

\bigskip
	
	
	\begin{abstract}
     In this paper, we research an anisotropic elliptic equation in the entire Euclidean space $\mathbb{R}^{N}$. Under the assumption that the nonlinearity of the equation has exponential or critical growth, we prove, by means of the mountain-pass theorem, that the equation admits a nontrivial positive weak solution. Moreover, applying the mountain-pass theorem and Ekeland's variational principle, we show that a perturbation of the equation possesses two distinct positive weak solutions.

     \vskip 0.2 true cm

	\end{abstract}

    \section{Introduction and main results}
    Let $\Omega\subset\mathbb{R}^{N}$ be a bounded smooth domain and consider the following problem:
    \begin{equation}\label{(1.1)}
        \left\{
        \begin{array}{ll}
            -\Delta_{p}u=f(x,u) \quad \text{in}\ \Omega,\vspace{1ex}\\
            u\in W_{0}^{1,p}(\Omega),
        \end{array}
        \right.
    \end{equation}
    where $\Delta_{p}u=\mathrm{div}(|\nabla u|^{p-2}\nabla u)$. In the case $p=2$ and $|f(x,u)|\leq c(|u|+|u|^{q-1})$ with $1<q\leq 2^{*}=\frac{2N}{N-2}$, $N\geq 3$, the existence theory for \eqref{(1.1)} was considered by many researchers such as Bartsch and Willem \cite{1}, Br\'{e}zis \cite{2}, Br\'{e}zis and Nirenberg \cite{3}, and Capozzi, Fortunato and Palmieri \cite{4}. When $p\leq N$ and $p^{2}\leq N$, the Br\'{e}zis-Nirenberg type existence and nonexistence results were carried over to the $p$-Laplacian by Garcia and Alonso \cite{5}. When the domain is replaced by the whole space $\mathbb{R}^{N}$ and $p=2$, the existence of the following semilinear Schr\"{o}dinger equation:
    \begin{equation}\label{(1.2)}
        \left\{
        \begin{array}{ll}
            -\Delta u+V(x)u=f(x,u) \quad \text{in}\ \mathbb{R}^{N},\vspace{1ex}\\
            u\in W^{1,2}(\mathbb{R}^{N}),
        \end{array}
        \right.
    \end{equation}
    where $|f(x,u)|\leq c(|u|+|u|^{q-1})$ with $1<q\leq 2^{*}=\frac{2N}{N-2}$ was researched by Alama and Li \cite{6}, Ding and Ni \cite{7}, Jeanjean \cite{8}, and Kryszewski and Szulkin \cite{9}. When $p=N$ and $f(x,u)$ behaves like $e^{\alpha |u|^{\frac{N}{N-1}}}$ as $|u|\rightarrow +\infty$, the Sobolev critical exponent is no longer available, and the sharp exponential integrability is provided by the Moser-Trudinger inequality \cite{16,17,18}; in this critical growth range, problem \eqref{(1.2)} was studied by Adimurthi \cite{10}, Adimurthi and Yadava \cite{11}, Ruf et al. \cite{12,13}, J. M. do \'{O} \cite{14}, and Panda \cite{15}.

    We next turn to the singular problem in the whole space:
    \begin{equation}\label{(1.3)}
        -\Delta_{N}u+V(x)|u|^{N-2}u=\frac{f(x,u)}{|x|^{\delta}},\quad x\in\mathbb{R}^{N},
    \end{equation}
    where $\Delta_{N}u=\mathrm{div}(|\nabla u|^{N-2}\nabla u)$, $V:\mathbb{R}^{N}\mapsto\mathbb{R}$ is a continuous function, $f:\mathbb{R}^{N}\times\mathbb{R}\mapsto\mathbb{R}$ is continuous and behaves like $e^{\alpha t^{\frac{N}{N-1}}}$ as $t\rightarrow +\infty$, and $0\leq\delta<N$.
    Equations of this type are used in the precise description of electric, gravitational and fluid potentials, and they appear in a variety of scientific domains such as fluid dynamics, electromagnetism and astronomy. For nonlinearities with subcritical or critical growth, \eqref{(1.3)} can be treated within the framework of general critical point theory. In the nonsingular case $\delta=0$, D. Cao \cite{19} solved \eqref{(1.3)} for $N=2$, and the result was extended to arbitrary dimensions by Panda \cite{20}, J. M. do \'{O} \cite{21} and Alves-Figueiredo \cite{22}; in the singular case $0<\delta<N$, Yang \cite{23} proved the existence of a positive solution to \eqref{(1.3)}.

    The mathematician also concerned the following perturbed version of \eqref{(1.3)}:
    \begin{equation}\label{(1.4)}
         -\Delta_{N}u+V(x)|u|^{N-2}u=\frac{f(x,u)}{|x|^{\delta}}+\epsilon h,\quad x\in\mathbb{R}^{N},
    \end{equation}
    where $\epsilon>0$ is a constant and $h:\mathbb{R}^{N}\mapsto\mathbb{R}$ belongs to the dual space of $W^{1,N}(\mathbb{R}^{N})$.
    Compared with \eqref{(1.3)}, the perturbation term $\epsilon h$ in \eqref{(1.4)} breaks the homogeneous structure of the energy functional, as a consequence, the problem may admit more than one solution. When $\delta=0$, J. M. do \'{O}, Medeiros and Severo \cite{24} proved the existence of solutions to \eqref{(1.4)} by combining Ekeland's variational principle, the mountain-pass theorem and a suitable Moser-Trudinger type inequality. When $0<\delta<N$, Adimurthi and Yang \cite{25} showed that \eqref{(1.4)} possesses two weak solutions, and Yang \cite{23} proved that these two solutions are distinct.

     When the medium under consideration is anisotropic, the Laplace operator is naturally replaced by the Finsler Laplacian, also called the anisotropic Laplacian, which is defined by
    $$Qu:=\sum_{i=1}^{N}\frac{\partial}{\partial x_{i}} (F(\nabla u)F_{\xi_{i}}(\nabla u)),$$
    where $F\in\mathcal{C}^{2}(\mathbb{R}^{N}\backslash\{0\})$ is a convex and positively homogeneous function of degree one and $F_{\xi_{i}}=\frac{\partial F}{\partial \xi_{i}}$.
    This operator is intimately connected with the Wulff shape of $F$, a notion that goes back to Wulff's classical work \cite{26} on the equilibrium geometry of crystals and the minimization of anisotropic surface tensions at the beginning of the last century. The anisotropic Sobolev framework associated with $Q$ has been developed extensively in the recent literature; we refer the reader to \cite{27,28,29,30,31,32,33,34} and the references therein for more information.

    In this paper, we study the anisotropic ellipitic equation as follows:
    \begin{equation}\label{(1.5)}
        -Q_{N}u+V(x)|u|^{N-2}u=\frac{f(x,u)}{F^{0}(x)^{\delta}}, \quad x\in\mathbb{R}^{N}(N\geq2),
    \end{equation}
     where $Q_{N}$, known as the $N$-Finsler Laplacian or the $N$-anisotropic Laplacian, is defined by $$Q_{N}u:=\sum_{i=1}^{N}\frac{\partial}{\partial x_{i}} (F^{N-1}(\nabla u)F_{\xi_{i}}(\nabla u)),$$
    $V:\mathbb{R}^{N}\mapsto\mathbb{R}$ is a continuous function, $f:\mathbb{R}^{N}\times\mathbb{R}\mapsto\mathbb{R}$ is continuous and behaves like $e^{\lambda t^{\frac{N}{N-1}}}$ as $t\to +\infty$, and $0\leq \delta <N$.
    Note that, when $0<\delta<N$, equation \eqref{(1.5)} is singular with respect to the weight $F^{0}(x)^{\delta}$; when $\delta=0$ it reduces to the nonsingular equation
    \begin{equation*}
        -Q_{N}u+V(x)|u|^{N-2}u=f(x,u), \quad x\in\mathbb{R}^{N}(N\geq2).
    \end{equation*}
    Both situations $\delta=0$ and $0<\delta<N$ are considered in this paper. Throughout this paper, a weak solution $u$ is called positive if $u(x)\geq 0$ a.e. in $\mathbb{R}^{N}$.

    We now collect the standing assumptions on $F$ and $F^{0}$ which were used throughout the paper. Let $F:\mathbb{R}^{N}\mapsto\mathbb{R}$ be a nonnegative convex function of class $\mathcal{C}^{2}(\mathbb{R}^{N}\backslash\{0\})$, which is even and positively homogeneous of degree one, i.e. $F(k\xi)=|k|F(\xi)$ for all $\xi\in\mathbb{R}^{N}$ and $k\in\mathbb{R}$, and suppose that $F(\xi)>0$ for every $\xi\neq 0$. By the homogeneity of $F$, there exist two positive constants $0<a\leq b<\infty$ such that $a|\xi|\leq F(\xi)\leq b|\xi|$ for every $\xi\in\mathbb{R}^{N}$.

    We define the map $\Psi: \mathbb{S}^{N-1}\mapsto \mathbb{R}^{N}$ by $\Psi(\xi)=F_{\xi}(\xi)$.
   Its image $\Psi(\mathbb{S}^{N-1})$ is a smooth convex hypersurface in $\mathbb{R}^{N}$, called the Wulff shape of $F$.
   Let $F^{0}(x):=\sup\limits_{\xi\in K}\langle x,\xi \rangle$ be the support function of $K:=\{x\in\mathbb{R}^{N}: F(x)\leq 1\}$.
   One can check that $F^{0}: \mathbb{R}^{N}\mapsto [0,+\infty)$ is a convex, homogeneous function of class $\mathcal{C}^{2}(\mathbb{R}^{N}\backslash\{0\})$ as well.
   In fact, $F^{0}$ and $F$ are dual to each other in the sense that $$F^{0}(x)= \sup\limits_{\xi\in\mathbb{R}^{N}\backslash\{0\}}\frac{\langle x,\xi \rangle}{F(\xi)},\quad F(x)= \sup\limits_{\xi\in\mathbb{R}^{N}\backslash\{0\}}\frac{\langle x,\xi \rangle}{F^{0}(\xi)}. $$
   It is readily seen that $\Psi(\mathbb{S}^{N-1})=\{x\in\mathbb{R}^{N}: F^{0}(x)=1\}$. We write $\mathcal{W}_{r}(x_{0}):=\{x\in\mathbb{R}^{N}: F^{0}(x-x_{0})\leq r\}$ for the Wulff ball of radius $r$ centered at $x_{0}$, and abbreviate $\mathcal{W}_{r}(0)$ as $\mathcal{W}_{r}$.
   Let $\kappa_{N}=|\mathcal{W}_{1}|=|\{x\in\mathbb{R}^{N}: F^{0}(x)\leq 1\}|$, where $|\cdot|$ denotes the Lebesgue measure; thus $\kappa_{N}$ is the volume of the unit Wulff ball in $\mathbb{R}^{N}$.

   We also list some properties of $F$ and $F^{0}$; for further information, see \cite{35,36,37}.
       \begin{enumerate}
           \item [\rm(a)] $|F(x)-F(y)|\leq F(x+y)\leq F(x)+F(y);$
           \item [\rm(b)] $\frac{1}{C}\leq |\nabla F(x)|\leq C, \ \frac{1}{C}\leq |\nabla F^{0}(x)|\leq C$ for some $C>0$ and $x\neq 0;$
           \item [\rm(c)] $\langle x,\nabla F(x)\rangle=F(x), \ \langle x,\nabla F^{0}(x)\rangle=F^{0}(x)$ for $x\neq 0;$
           \item [\rm(d)] $F(\nabla F^{0}(x))=1, \ F^{0}(\nabla F(x))=1$ for $x\neq 0;$
           \item [\rm(e)] $F^{0}(x)F_{\xi}(\nabla F^{0}(x))=x, \ F(\xi)F_{x}^{0}(\nabla F(\xi))=\xi$ for any $x\neq 0, \ \xi\neq 0;$
           \item [\rm(f)] $F_{\xi}(t\xi)=sgn(t)F_{\xi}(\xi)$ for any $\xi\neq 0$ and $t\neq 0.$
       \end{enumerate}

    We assume that the potential $V$ satisfies the following conditions:

   \noindent $(A_{1})$ \quad There exists a positive constant $V^{*}$ such that $V(x)\geq V^{*}>0$ for every $x\in\mathbb{R}^{N}$;

    \noindent $(A_{2})$ \quad The function $\frac{1}{V(x)}$ belongs to $L^{\frac{1}{N-1}}(\mathbb{R}^{N})$.

     We define the function space
    $$S=\left\{ u\in W^{1,N}(\mathbb{R}^{N}):\int_{\mathbb{R}^{N}}V(x)|u| ^{N}dx<\infty \right\} ,$$
    endowed with the norm
    $$\|u\|_{S}=\left ( \int_{\mathbb{R}^{N}}(F^{N}(\nabla u)+V(x)|u|^{N})dx\right )^{\frac{1}{N}}. $$
    In view of $(A_{1})$, $S$ is a reflexive Banach space. Combining $(A_{1})$ with the Sobolev embedding theorem, we see that the embedding $S\hookrightarrow W^{1,N}(\mathbb{R}^{N})\hookrightarrow L^{q}(\mathbb{R}^{N})$ is continuous for every $q\geq N$.

    Throughout the paper, the nonlinearity $f(x,t)$ is assumed to satisfy the following conditions:

    \noindent $(B_{1})$ \quad There exist constants $\lambda_{0}, a_{1}, a_{2}>0$ such that, for every $(x,t) \in \mathbb{R}^{N} \times \mathbb{R}^{+}$,
    $$|f(x,t)|\leq a_{1}|t|^{N-1}+a_{2}\Phi\left(N,\lambda_{0}|t|^{\frac{N}{N-1}}\right),$$
    where $$\Phi(N,k)=e^{k}-\sum_{j=0}^{N-2}\frac{k^{j}}{j!};$$

    \noindent $(B_{2})$ \quad  There exists $\eta>N$ such that, for every $x\in \mathbb{R}^{N}$ and $t>0$,
    $$0<\eta G(x,t)\equiv \eta \int_{0}^{t}f(x,k)dk\leq tf(x,t) ;$$

    \noindent $(B_{3})$ \quad There exist constants $T_{0},R_{0}>0$ such that, for every $x\in \mathbb{R}^{N}$ and $t\geq T_{0}$,
    $$G(x,t)\leq R_{0}f(x,t);$$

    \noindent $(B_{4})$ \quad $\limsup\limits_{t\to 0^{+}}\frac{N|G(x,t)|}{t^{N}}<\lambda_{\delta}$ uniformly with respect to $x\in \mathbb{R}^{N}$, where $\lambda_{\delta}$ is the singular eigenvalue of the $N$-anisotropic Laplacian defined by $$\lambda_{\delta}=\inf_{u\in S,u\not\equiv 0}\frac{\|u\|_{S}^{N}}{\int_{\mathbb{R}^{N}}\frac{|u|^{N}}{F^{0}(x)^{\delta}}dx}.$$
    The continuous embedding $S\hookrightarrow L^{q}(\mathbb{R}^{N})$ for $q\geq N$ together with the H\"{o}lder inequality implies
    \begin{equation*}
        \int_{\mathbb{R}^{N}}\frac{|u|^{N}}{F^{0}(x)^{\delta}}dx\leq \int_{F^{0}(x)>1}|u|^{N}dx+\left(\int_{F^{0}(x)\leq1}|u|^{Np_{1}}dx\right)^{\frac{1}{p_{1}}}\left(\int_{F^{0}(x)\leq1}\frac{1}{F^{0}(x)^{\delta p_{2}}}\right)^{\frac{1}{p_{2}}}\leq C\|u\|_{S}^{N},
    \end{equation*}
    where $\frac{1}{p_{1}}+\frac{1}{p_{2}}=1$, $0<\delta p_{2}<N$, and $C$ is a positive constant. Hence $\lambda_{\delta}>0$;

    \noindent $(B_{5})$ \quad $\liminf\limits_{t\to +\infty}tf(x,t)e^{-\lambda_{0}t^\frac{N}{N-1}}=\beta_{0}>\mathcal{B}$ uniformly with respect to $x\in \mathbb{R}^{N}$, where
    \begin{equation*}
    \mathcal{B}=\inf\limits_{l>0}\frac{(N-\delta)^{N}}{\lambda_{0}^{N-1}l^{N-\delta}}e^{(N-\delta)g(l)\frac{(N-2)!}{N^{N}}l^{N}},
    \end{equation*}
    $g(l)=\sup\limits_{F^{0}(x)\leq l}V(x)$ and $\lambda_{0}$ is given by $(B_{1})$.
    Since $V$ is continuous and $(A_{1})$ holds, $g(l)$ is a positive continuous function of $l$, and the infimum $\mathcal{B}$ is attained at some $l>0$.

    A function $u\in S$ is called a weak solution of problem \eqref{(1.5)} if
    $$\int_{\mathbb{R}^{N}}(F^{N-1}(\nabla u)F_{\xi}(\nabla u)\nabla \phi+V(x)|u|^{N-2}u\phi)dx=\int_{\mathbb{R}^{N}}\frac{f(x,u)}{F^{0}(x)^{\delta}}\phi dx$$
    holds for every $\phi \in S$.

    Along with \eqref{(1.5)}, we consider its perturbation
    \begin{equation}\label{(1.6)}
        -Q_{N}u+V(x)|u|^{N-2}u=\frac{f(x,u)}{F^{0}(x)^{\delta}}+\epsilon h, \quad x\in\mathbb{R}^{N}(N\geq2),
    \end{equation}
    where $\epsilon>0$ is a constant and $h:\mathbb{R}^{N}\mapsto\mathbb{R}$ belongs to $S^{*}$, the dual space of $S$.
    One can refer to Xie and Yang \cite{44} about the related problems on the bounded domain.

    We are now in a position to state the main results of this paper.

    \begin{theorem}\label{Th1.1}
     Suppose that $(A_{1})$, $(A_{2})$ and $(B_{1})-(B_{5})$ hold. Then problem \eqref{(1.5)} admits a nontrivial positive weak solution of mountain-pass type.
    \end{theorem}

    \begin{theorem}\label{Th1.2}
      Suppose that $(A_{1})$, $(A_{2})$ and $(B_{1})-(B_{5})$ hold. Furthermore, assume that $h(x)\geq 0$ and $h(x)\not\equiv 0$. Then there exists $\epsilon^{*}>0$ such that, for every $0<\epsilon<\epsilon^{*}$, problem \eqref{(1.6)} has two distinct positive weak solutions.
    \end{theorem}

    We may assume throughout this paper that $f(x,t)\equiv 0$ for every $t<0$. Set
    \begin{equation*}
        f^{*}(x,t)=\left\{
        \begin{array}{ll}
            f(x,t), &f(x,t)\geq 0,\vspace{1ex}\\
            0, &f(x,t)<0.
        \end{array}
        \right.
    \end{equation*}
    If $u\in S$ is a weak solution of
    \begin{equation}\label{(1.7)}
        -Q_{N}u+V(x)|u|^{N-2}u=\frac{f^{*}(x,u)}{F^{0}(x)^{\delta}}+\epsilon h,
    \end{equation}
     where $h\geq 0$ and $\epsilon>0$, then the negative part of $u$,
    \begin{equation*}
        u^{-}(x)=\left\{
        \begin{array}{ll}
           u(x), &u(x)\leq 0, \vspace{1ex} \\
           0, &u(x)>0
        \end{array}
        \right.
    \end{equation*}
    belongs to $S$ and satisfies
    \begin{align*}
       \int_{\mathbb{R}^{N}}(F^{N}(\nabla u^{-})+V(x)|u^{-}|^{N})dx
       &=\int_{\mathbb{R}^{N}}\frac{f^{*}(x,u)}{F^{0}(x)^{\delta}}u^{-}dx+\epsilon\int_{\mathbb{R}^{N}}hu^{-}dx\\
&=\epsilon\int_{\mathbb{R}^{N}}hu^{-}dx\leq 0.
    \end{align*}
    Hence $u^{-}(x)=0$ a.e. in $\mathbb{R}^{N}$, so $u$ is a positive weak solution of \eqref{(1.7)}. In view of $(B_{2})$, we have $f(x,u)\geq 0$; consequently $f^{*}(x,u)=f(x,u)$, and $u$ is also a positive weak solution of \eqref{(1.6)}. When $h=0$, equation \eqref{(1.6)} reduces to \eqref{(1.5)}.
     Therefore, in order to prove Theorems \ref{Th1.1} and \ref{Th1.2}, it suffices to find weak solutions of \eqref{(1.5)} and \eqref{(1.6)} with $f$ replaced by $f^{*}$, respectively.

     This paper is organized as follows. In Section 2, we collect some preliminaries and inequalities that will be used later. In Section 3, we introduce the functionals associated with problems \eqref{(1.5)} and \eqref{(1.6)} and establish their main properties. In Section 4, we prove Theorem \ref{Th1.1} by the mountain-pass theorem, then prove Theorem \ref{Th1.2} by combining the mountain-pass theorem with Ekeland's variational principle and by showing that the two solutions obtained are distinct.

     Throughout this paper, the letter $C$ will denote various positive constants whose exact values may change from line to line. In what follows, we do not distinguish between sequences and subsequences.

     \section{Preliminaries}

     First, we recall the anisotropic singular Moser-Trudinger type inequality in a bounded domain, which was established by Zhu \cite{39} by means of a convex symmetrization technique \cite{27}, namely
     \begin{lemma}\label{Le2.1}
     Let $\Omega$ be a bounded domain in $\mathbb{R}^{N}$ which contains the origin, for every $\lambda >0, 0\leq\delta<N, u\in W^{1,N}_{0}(\Omega)(N\geq 2)$, there holds
       \begin{equation*}
    \int_{\Omega}\frac{e^{\lambda|u|^{\frac{N}{N-1}}}}{F^{0}(x)^{\delta}}dx<\infty.
            \end{equation*}
    Furthermore, we have for every $\lambda\leq(1-\frac{\delta}{N})\lambda_{N}$,
    \begin{equation*}
          \sup\limits_{\int_{\Omega}F^{N}(\nabla u)dx\leq 1}\int_{\Omega}\frac{e^{\lambda|u|^{\frac{N}{N-1}}}}{F^{0}(x)^{\delta}}dx<\infty,
       \end{equation*}
       where $\lambda_{N}=N^{\frac{N}{N-1}}\kappa_{N}^{\frac{1}{N-1}}$, and this inequality is sharp as the supremum is infinity for every $\lambda>\lambda_{N}$.
   \end{lemma}

   Zhu \cite{39} also extended the above inequality to the entire Euclidean space $\mathbb{R}^{N}$, namely
   \begin{lemma}\label{Le2.2}
       For every $\lambda >0, 0\leq\delta<N, \tau>0, u\in W^{1,N}(\mathbb{R}^{N})(N\geq 2)$, there holds
       \begin{equation}\label{(2.1)}
\int_{\mathbb{R}^{N}}\frac{\Phi(N,\lambda|u|^{\frac{N}{N-1}})}{F^{0}(x)^{\delta}}dx<\infty.
       \end{equation}
       Furthermore, we have for every $\lambda\leq(1-\frac{\delta}{N})\lambda_{N}$,
       \begin{equation}\label{(2.2)}
           \sup\limits_{\int_{\mathbb{R}^{N}}(F^{N}(\nabla u)+\tau|u|^{N})dx\leq 1}\int_{\mathbb{R}^{N}}\frac{\Phi(N,\lambda|u|^{\frac{N}{N-1}})}{F^{0}(x)^{\delta}}dx<\infty,
       \end{equation}
      where $\lambda_{N}=N^{\frac{N}{N-1}}\kappa_{N}^{\frac{1}{N-1}}$, and this inequality is sharp as the supremum is infinity for every $\lambda>\lambda_{N}$.
   \end{lemma}

   We shall also need the following anisotropic singular Lions-type concentration-compactness lemma in $\mathbb{R}^{N}$, which can be proved in similar way in \cite{40}.
   \begin{lemma}\label{Le2.3}
       Let $\{u_{j}\}$ be a sequence in $S$ such that $\|u_{j}\|_{S}=1$ and $u_{j}\rightharpoonup u\not\equiv 0$ weakly in $S$.
       If $$0<p<\frac{1}{(1-\|u\|_{S}^{N})^{\frac{1}{N-1}}}$$ then
       \begin{equation}\label{(2.3)}
           \limsup\limits_{j\to \infty}\int_{\mathbb{R}^{N}}\frac{\Phi(N,\lambda_{N}(1-\frac{\delta}{N})p|u_{j}|^{\frac{N}{N-1}})}{F^{0}(x)^{\delta}}dx<+\infty.
       \end{equation}
   \end{lemma}

   \begin{lemma}\label{Le2.4} \rm(See \cite{23})
       Let $k\geq 0$, $p\geq 1$ be real numbers and $N\geq 2$ be an integer. Then there holds
       \begin{equation}\label{(2.4)}
           (\Phi(N,k))^{p}\leq \Phi(N,pk).
       \end{equation}
   \end{lemma}

   The following lemma provides the compact embedding $S\hookrightarrow\hookrightarrow L^{q}(\mathbb{R}^{N})$ for every $q\geq 1$; it can be proved in the same way as Lemma 2.4 of \cite{23}.
   \begin{lemma}\label{Le2.5}
      Assume that $(A_{1})$ and $(A_{2})$ are satisfied. Then $S$ is compactly embedded in $L^{q}(\mathbb{R}^{N})$ for every $q\geq 1$.
  \end{lemma}

  Next, set
    \begin{equation*}
       d_{X,Y}=\frac{\langle F^{N-1}(X)F_{\xi}(X)-F^{N-1}(Y)F_{\xi}(Y), X-Y\rangle}{F^{N}(X-Y)},
    \end{equation*}
    \begin{equation*}
         d_{0}= \inf\{ d_{X,Y}:X,Y\in\mathbb{R}^{N}, X\neq 0, Y\neq 0, X\neq Y\},
    \end{equation*}
      then we obtain
      \begin{equation} \label{(2.5)}
         \langle F^{N-1}(X)F_{\xi}(X)-F^{N-1}(Y)F_{\xi}(Y), X-Y \rangle \geq d_{0}F^{N}(X-Y).
      \end{equation}

   \begin{lemma}\label{Le2.7}\rm(See \cite{41})
       \begin{equation}\label{(2.6)}
           \langle |c|^{N-2}c-|d|^{N-2}d, c-d \rangle \geq 2^{2-N}|c-d|^{N}, \quad \forall c, d\in\mathbb{R}^{N}(N\geq 2).
       \end{equation}
   \end{lemma}

   \section{Some results related to functionals}

  Suppose that $f$ satisfies $(B_{1})$. Then, for every $(x,t)\in \mathbb{R}^{N}\times\mathbb{R}$, there exist constants $\lambda_{1}>\lambda_{0}$ and $a_{3}>0$ such that $G(x,t)\leq a_{3}\Phi\left(N,\lambda_{1}|t|^{\frac{N}{N-1}}\right)$. Hence, by Lemma \ref{Le2.2}, $\frac{G(x,u)}{F^{0}(x)^{\delta}}\in L^{1}(\mathbb{R}^{N})$ for every $u\in W^{1,N}(\mathbb{R}^{N})$. We now define the functionals $E, E_{\epsilon}: S\mapsto \mathbb{R}$ by
   \begin{equation}\label{(3.1)}
     E(u)=\frac{1}{N}\|u\|_{S}^{N}-\int_{\mathbb{R}^{N}}\frac{G(x,u)}{F^{0}(x)^{\delta}}dx,
   \end{equation}
    and
   \begin{equation}\label{(3.2)}
     E_{\epsilon}(u)=\frac{1}{N}\|u\|_{S}^{N}-\int_{\mathbb{R}^{N}}\frac{G(x,u)}{F^{0}(x)^{\delta}}dx-\epsilon\int_{\mathbb{R}^{N}}hudx.
   \end{equation}

   In order to apply critical point theory, we need the regularity of the functionals $E$ and $E_{\epsilon}$. From Proposition 1 in \cite{24} and standard arguments in \cite{42}, both $E$ and $E_{\epsilon}$ belong to $\mathcal{C}^{1}(S,\mathbb{R})$.
   A straightforward computation gives
   \begin{equation}\label{(3.3)}
       \langle E^{'}(u), \psi \rangle =\int_{\mathbb{R}^{N}}(F^{N-1}(\nabla u)F_{\xi}(\nabla u)\nabla \psi+V(x)|u|^{N-2}u\psi)dx-\int_{\mathbb{R}^{N}}\frac{f(x,u)}{F^{0}(x)^{\delta}}\psi dx,
   \end{equation}

   \begin{equation}\label{(3.4)}
    \begin{aligned}
       \langle E^{'}_{\epsilon}(u), \psi \rangle &=\int_{\mathbb{R}^{N}}(F^{N-1}(\nabla u)F_{\xi}(\nabla u)\nabla \psi+V(x)|u|^{N-2}u\psi)dx
       -\int_{\mathbb{R}^{N}}\frac{f(x,u)}{F^{0}(x)^{\delta}}\psi dx\\
       &\quad-\epsilon\int_{\mathbb{R}^{N}}h\psi dx
    \end{aligned}
   \end{equation}
   for all $\psi\in S$.
   Therefore, weak solutions of \eqref{(1.5)} and \eqref{(1.6)} correspond exactly to critical points of $E$ and $E_{\epsilon}$, respectively.

   In the next three lemmas, we establish the geometric conditions of the mountain-pass theorem for the functionals $E$ and $E_{\epsilon}$.

   \begin{lemma}\label{Le3.1}
       Suppose that $(A_{1})$, $(B_{1})$, $(B_{2})$ and $(B_{3})$ are satisfied. Then there holds $E_{\epsilon}(ku)\to -\infty$ as $k\to +\infty$, for every nonnegative, compactly supported function $u\in W^{1,N}(\mathbb{R}^{N})\backslash\{0\}$.
   \end{lemma}

   \begin{proof}
       By $(B_{2})$ and $(B_{3})$, there exists $T_{0}>0$ such that $G(x,t)>0$ and $\eta G(x,t)\leq t \frac{\partial }{\partial t} G(x,t)$ for every $(x,t)\in\mathbb{R}^{N}\times[T_{0},\infty)$. It follows that $\frac{\partial }{\partial t}(\ln G(x,t))\geq \frac{\eta}{t}$, and hence $G(x,t)\geq G(x,T_{0})T_{0}^{-\eta}t^{\eta}$. Suppose that $u$ is supported in a bounded domain $\Omega$. Then there exist $b_{1}, b_{2}>0$ such that $G(x,t)\geq b_{1}t^{\eta}-b_{2}$ for every $(x,t)\in \Omega\times[0,\infty)$. We have
        \begin{align*}
            E_{\epsilon}(ku)
            &=\frac{1}{N}\|ku\|_{S}^{N}-\int_{\Omega}\frac{G(x,ku)}{F^{0}(x)^{\delta}}dx-\epsilon k\int_{\Omega}hu dx\\
            &\leq \frac{1}{N}k^{N}\|u\|_{S}^{N}-b_{1}k^{\eta}\int_{\Omega}\frac{|u|^{\eta}}{F^{0}(x)^{\delta}}dx+O(k).
        \end{align*}
        Since $\eta>N$, we conclude that $E_{\epsilon}(ku)\to -\infty$ as $k\to +\infty$. This completes the proof of Lemma \ref{Le3.1}.
   \end{proof}

   \begin{lemma}\label{Le3.2}
       Suppose that $(A_{1})$, $(B_{1})$ and $(B_{4})$ are satisfied. Then there exists $\epsilon_{1}>0$ such that when $0<\epsilon<\epsilon_{1}$, $l_{\epsilon}, \theta_{\epsilon}>0$ can be found such that $E_{\epsilon}(u)\geq \theta_{\epsilon}$ for every $u$ with $\|u\|_{S}=l_{\epsilon}$, where $l_{\epsilon}$ can be further chosen such that $l_{\epsilon}\to 0$ as $\epsilon\to 0$.
        When $\epsilon =0$, there exist $l>0$ and $\gamma >0$ such that $E(u)\geq\gamma$ for every $\|u\|_{S}=l$.
   \end{lemma}

   \begin{proof}
       By $(B_{4})$, there exist $\nu, \omega>0$ such that if $|t|< \omega$, then
        \begin{equation}\label{(3.5)}
            G(x,t)\leq \frac{\lambda_{\delta}-\nu}{N}|t|^{N}
        \end{equation}
        for every $x\in\mathbb{R}^{N}$.
        When $|t|\geq \omega$, it follows from $(B_{1})$ that
        \begin{align}\label{(3.6)}
            G(x,t)&\leq \int_{0}^{|t|}\left( a_{1}k^{N-1}+a_{2}\Phi\left(N,\lambda_{0}|k|^{\frac{N}{N-1}}\right)\right)dk \nonumber \\
            &\leq \frac{a_{1}}{N}|t|^{N}+a_{2}\Phi\left(N,\lambda_{0}|t|^{\frac{N}{N-1}}\right)|t| \nonumber \\
            &\leq c_{\omega}|t|^{N+1}\Phi\left(N,\lambda_{0}|t|^{\frac{N}{N-1}}\right),
        \end{align}
        where $c_{\omega}=\frac{a_{1}}{\omega N\Phi\left(N,\lambda_{0}|\omega|^{\frac{N}{N-1}}\right)}+\frac{a_{2}}{\omega^{N}}$.
        Combining \eqref{(3.5)} and \eqref{(3.6)}, we obtain, for every $(x,t)\in \mathbb{R}^{N}\times \mathbb{R}$,
        \begin{equation}\label{(3.7)}
            G(x,t)\leq \frac{\lambda_{\delta}-\nu}{N}|t|^{N}+C|t|^{N+1}\Phi\left(N,\lambda_{0}|t|^{\frac{N}{N-1}}\right).
        \end{equation}
        By the H\"{o}lder inequality, Lemmas \ref{Le2.4} and \ref{Le2.2}, and the continuous embedding $S\hookrightarrow L^{q}$ for every $q\geq 1$, we have
        \begin{equation}\label{(3.8)}
            \int_{\mathbb{R}^{N}}
          \frac{|u|^{N+1}\Phi\left(N,\lambda_{0}|u|^{\frac{N}{N-1}}\right)}{F^{0}(x)^{\delta}}dx\leq C\|u\|_{S}^{N+1}.
        \end{equation}
        With the help of \eqref{(3.2)}, \eqref{(3.7)} and \eqref{(3.8)}, we obtain
        \begin{align}\label{(3.9)}
            E_{\epsilon}(u)&=\frac{1}{N}\|u\|_{S}^{N}-\int_{\mathbb{R}^{N}}\frac{G(x,u)}{F^{0}(x)^{\delta}}dx-\epsilon\int_{\mathbb{R}^{N}}hudx \nonumber\\
            &\geq \frac{1}{N}\|u\|_{S}^{N}-\frac{\lambda_{\delta}-\nu}{N}\int_{\mathbb{R}^{N}}\frac{|u|^{N}}{F^{0}(x)^{\delta}}dx-C\|u\|_{S}^{N+1}-\epsilon\int_{\mathbb{R}^{N}}hudx \nonumber\\
            &\geq \frac{\nu}{N\lambda_{\delta}}\|u\|_{S}^{N}-C\|u\|_{S}^{N+1}-\frac{\epsilon}{V^{*\frac{1}{N}}}\|h\|_{S^{*}}\|u\|_{S} \nonumber\\
            &=\|u\|_{S}\left(\frac{\nu}{N\lambda_{\delta}}\|u\|_{S}^{N-1}-C\|u\|_{S}^{N}-\frac{\epsilon}{V^{*\frac{1}{N}}}\|h\|_{S^{*}}\right).
        \end{align}
         Given $\nu>0$, for sufficiently small $l>0$, we have
        $$\frac{\nu}{N\lambda_{\delta}}l^{N-1}-Cl^{N}\geq \frac{\nu}{2N\lambda_{\delta}}l^{N-1}.$$
        For $\epsilon>0$ sufficiently small, we may choose $l_{\epsilon}$ and $\theta_{\epsilon}$ such that $$\frac{\nu}{2N\lambda_{\delta}}l_{\epsilon}^{N-1}=\frac{2\epsilon}{V^{*\frac{1}{N}}}\|h\|_{S^{*}}, \quad \theta_{\epsilon}=\frac{\epsilon}{V^{*\frac{1}{N}}} l_{\epsilon}\|h\|_{S^{*}}.$$
        Then $E_{\epsilon}(u)\geq \theta_{\epsilon}$ for every $u$ with $\|u\|_{S}=l_{\epsilon}$, and $l_{\epsilon}\to 0$ as $\epsilon\to 0$.
        If $\epsilon=0$, then by \eqref{(3.9)},
        \begin{equation*}
            E(u)\geq \|u\|_{S}\left(\frac{\nu}{N\lambda_{\delta}}\|u\|_{S}^{N-1}-C\|u\|_{S}^{N}\right).
        \end{equation*}
        Hence there exist $l>0$ and $\gamma=\frac{\nu}{2N\lambda_{\delta}}l^{N}>0$ such that $E(u)\geq\gamma$ for every $\|u\|_{S}=l$. This completes the proof.
   \end{proof}

   \begin{lemma}\label{Le3.3}
       Suppose that $\epsilon>0$, $h\not\equiv 0$ and $(B_{1})$ are satisfied. Then there exists a constant $\rho >0$ and $w\in S$ with $\|w\|_{S}=1$ such that if $0<k<\rho$, then $E_{\epsilon}(kw)<0$. Particularly, $\inf\limits_{\|u\|_{S}\leq \rho}E_{\epsilon}(u)<0$.
   \end{lemma}

   \begin{proof}
       Let $w\in S$ be a solution of
        \begin{equation*}
            -Q_{N}w+V(x)|w|^{N-2}w=h \quad \text{in} \ \mathbb{R}^{N}.
        \end{equation*}
         The existence of $w$ follows from the direct method in the calculus of variations. Since $h\not\equiv 0$, we have
        \begin{equation*}
            \int_{\mathbb{R}^{N}}(F^{N-1}(\nabla w)F_{\xi}(\nabla w)\nabla w+V(x)|w|^{N})dx=\int_{\mathbb{R}^{N}}hwdx,
        \end{equation*}
        hence
        \begin{equation*}
          \int_{\mathbb{R}^{N}}hwdx=\|w\|_{S}^{N}>0.
        \end{equation*}
        For $k>0$, it follows that
        \begin{equation*}
            \frac{\mathrm{d} }{\mathrm{d} k}E_{\epsilon}(kw)=k^{N-1}\|w\|_{S}^{N}-\int_{\mathbb{R}^{N}}\frac{f(x,kw)w}{F^{0}(x)^{\delta}}dx-\epsilon\int_{\mathbb{R}^{N}}hwdx.
        \end{equation*}
        In view of $f(x,0)=0$, we have $\frac{\mathrm{d} }{\mathrm{d} k}E_{\epsilon}(kw)\big|_{k=0}<0$.
        By continuity, there exists $\rho>0$ such that, for every $k$ with $0<k<\rho$,
        \begin{equation*}
             \frac{\mathrm{d} }{\mathrm{d} k}E_{\epsilon}(kw)<0.
        \end{equation*}
        Since $E_{\epsilon}(0)=0$, we obtain $E_{\epsilon}(kw)<0$ for every $0<k<\rho$. The proof of Lemma \ref{Le3.3} is completed.
   \end{proof}

   In order to obtain more precise information on the min-max values of the functionals $E$ and $E_{\epsilon}$, we consider the following sequence of test functions introduced by Moser \cite{16}. Set
   \begin{equation*}
         \tilde{v_{j}}(x,l)=\frac{1}{(N\kappa _{N})^{\frac{1}{N}}}
         \left\{
         \begin{aligned}
            &(\log j)^{1-\frac{1}{N}}, &&0<F^{0}(x)\leq\frac{l}{j},\\
            &\frac{\log\frac{l}{F^{0}(x)}}{(\log j)^{\frac{1}{N}}}, &&\frac{l}{j}<F^{0}(x)\leq l,\\
            &0, &&F^{0}(x)>l.
        \end{aligned}
        \right.
     \end{equation*}
     Let $v_{j}(x,l)=\frac{\tilde{v_{j}}(x,l)}{\|\tilde{v_{j}}\|_{S}}$. Then it is clear that $v_{j}\in S$, $\mathrm{supp}\,v_{j}\subset\mathcal{W}_{l}$ and $\|v_{j}\|_{S}=1$.

     With the similar way to prove Lemma 3.2 in \cite{23}, we can obtain the following result. Here we omit the details.
    \begin{lemma}\label{Le3.4}
         Suppose that $V$ is continuous and satisfies $(A_{1})$. Then
         \begin{equation}\label{(3.10)}
             \|\tilde{v_{j}}\|_{S}^{N}\leq 1+\frac{g(l)}{\log j}\left(\frac{(N-1)!}{N^{N}}l^{N}+o_{j}(1)\right),
         \end{equation}
         where $g(l)=\sup\limits_{F^{0}(x)\leq l}V(x)$ and $o_{j}(1)\to 0$ as $j\to \infty$.
     \end{lemma}

     \begin{lemma}\label{Le3.5}
         Suppose that $(A_{1})$, $(B_{1})$, $(B_{2})$, $(B_{3})$ and $(B_{5})$ hold. Then there exists $j\in\mathbb{N}$ such that
        \begin{equation}\label{(3.11)}
            \max\limits_{k\geq 0}E(kv_{j})<\frac{1}{N}\left(\frac{N-\delta}{N}\frac{\lambda_{N}}{\lambda_{0}}\right)^{N-1}.
        \end{equation}
        Moreover, for this $j$, there exist $\epsilon^{0}>0$ and $\gamma^{0}>0$ such that, if $0\leq\epsilon<\epsilon^{0}$, then
        \begin{equation}\label{(3.12)}
            \max\limits_{k\geq 0}E_{\epsilon}(kv_{j})<\frac{1}{N}\left(\frac{N-\delta}{N}\frac{\lambda_{N}}{\lambda_{0}}\right)^{N-1}-\gamma^{0}.
        \end{equation}
     \end{lemma}

     \begin{proof}
         By $(B_{5})$, there exists $l>0$ such that
         \begin{equation}\label{(3.13)}
             \beta_{0}>\frac{(N-\delta)^{N}}{\lambda_{0}^{N-1}l^{N-\delta}}e^{(N-\delta)g(l)\frac{(N-2)!}{N^{N}}l^{N}}.
         \end{equation}
      First, we prove \eqref{(3.11)} by contradiction. Assume that, for every $j\in\mathbb{N}$, we have
         \begin{equation}\label{(3.14)}
            \max\limits_{k\geq 0}E(kv_{j})\geq\frac{1}{N}\left(\frac{N-\delta}{N}\frac{\lambda_{N}}{\lambda_{0}}\right)^{N-1}.
         \end{equation}
         By Lemma \ref{Le3.1}, for every $j\in\mathbb{N}$, there exists $k_{j}>0$ such that $$E(k_{j}v_{j})=\max\limits_{k\geq 0}E(kv_{j}).$$
         Then by \eqref{(3.14)}, we obain
         \begin{align*}
            E(k_{j}v_{j})
            &=\frac{1}{N}\|k_{j}v_{j}\|_{S}^{N}-\int_{\mathbb{R}^{N}}\frac{G(x,k_{j}v_{j})}{F^{0}(x)^{\delta}}dx\\
            &=\frac{k_{j}^{N}}{N}-\int_{\mathbb{R}^{N}}\frac{G(x,k_{j}v_{j})}{F^{0}(x)^{\delta}}dx\\
            &\geq\frac{1}{N}\left(\frac{N-\delta}{N}\frac{\lambda_{N}}{\lambda_{0}}\right)^{N-1}.
         \end{align*}
         Since $G(x,\cdot)\geq 0$, we have
         \begin{equation}\label{(3.15)}
             k_{j}^{N}\geq \left(\frac{N-\delta}{N}\frac{\lambda_{N}}{\lambda_{0}}\right)^{N-1}.
         \end{equation}
         When $k=k_{j}$, we have
         $$\frac{\mathrm{d} }{\mathrm{d} k}\left(\frac{k^{N}}{N}-\int_{\mathbb{R}^{N}}\frac{G(x,kv_{j})}{F^{0}(x)^{\delta}}dx\right)=0.$$
         Then
         \begin{equation}\label{(3.16)}
          k_{j}^{N}=\int_{\mathbb{R}^{N}}\frac{k_{j}v_{j}
           f(x,k_{j}v_{j})}{F^{0}(x)^{\delta}}dx=\int_{\mathcal{W}_{l}}\frac{k_{j}v_{j}
           f(x,k_{j}v_{j})}{F^{0}(x)^{\delta}}dx.
         \end{equation}
          From $(B_{5})$, for every $\sigma>0$, there exists $R_{\sigma}>0$ such that for every $x\in\mathbb{R}^{N}$ and $u\geq R_{\sigma}$, it follows that
         \begin{equation}\label{(3.17)}
             uf(x,u)\geq (\beta_{0}-\sigma)e^{\lambda_{0}|u|^{\frac{N}{N-1}}}.
         \end{equation}
         When $0<F^{0}(x)\leq \frac{l}{j}$, by Lemma \ref{Le3.4}, we obtain
         \begin{align}\label{(3.18)}
             v_{j}^{\frac{N}{N-1}}(x,l)
             &\geq (N\kappa_{N})^{-\frac{1}{N-1}} \frac{\log j}{1+\frac{g(l)}{(N-1)\log j}\left(\frac{(N-1)!}{N^{N}}l^{N}+o_{j}(1)\right)} \nonumber \\
             &=(N\kappa_{N})^{-\frac{1}{N-1}}\log j
             -(N\kappa_{N})^{-\frac{1}{N-1}}g(l)\frac{(N-2)!}{N^{N}}l^{N}+o_{j}(1).
         \end{align}
         With the help of \eqref{(3.16)}, \eqref{(3.17)} and \eqref{(3.18)}, we have
         \begin{align}\label{(3.19)}
             k_{j}^{N}
             &\geq (\beta_{0}-\sigma)\int_{\mathcal{W}_{\frac{l}{j}}}\frac{e^{\lambda_{0}|k_{j}v_{j}|^{\frac{N}{N-1}}}}{F^{0}(x)^{\delta}}dx \nonumber \\
             &\geq (\beta_{0}-\sigma)\int_{\mathcal{W}_{\frac{l}{j}}}\frac{e^{\lambda_{0}(N\kappa_{N})^{-\frac{1}{N-1}}k_{j}^{\frac{N}{N-1}}(\log j-g(l)\frac{(N-2)!}{N^{N}}l^{N}+o_{j}(1))}}{F^{0}(x)^{\delta}}dx \nonumber \\
             &=(\beta_{0}-\sigma)\frac{N\kappa_{N}}{N-\delta}(\frac{l}{j})^{N-\delta}e^{\lambda_{0}(N\kappa_{N})^{-\frac{1}{N-1}}k_{j}^{\frac{N}{N-1}}(\log j-g(l)\frac{(N-2)!}{N^{N}}l^{N}+o_{j}(1))},
         \end{align}
         which shows in particular that the sequence $(k_{j})$ is bounded below. A direct computation then yields
      \begin{equation}\label{(3.20)}
      \lim\limits_{j\to\infty}k_{j}^{N}
       =\left(\frac{N-\delta}{N}\frac{\lambda_{N}}{\lambda_{0}}\right)^{N-1}.
      \end{equation}

     Set
     \begin{equation*}
         A_{j}=\{x\in\mathcal{W}_{l}:k_{j}v_{j}\geq R_{\sigma}\} \quad \mathrm{and} \quad B_{j}=\mathcal{W}_{l}\backslash A_{j}.
     \end{equation*}
     By \eqref{(3.16)} and \eqref{(3.17)}, we have
     \begin{equation}\label{(3.21)}
    \begin{aligned}
        k_{j}^{N}
        &\geq (\beta_{0}-\sigma)\int_{\mathcal{W}_{l}}\frac{e^{\lambda_{0}|k_{j}v_{j}|^{\frac{N}{N-1}}}}{F^{0}(x)^{\delta}}dx-(\beta_{0}-\sigma)\int_{B_{j}}\frac{e^{\lambda_{0}|k_{j}v_{j}|^{\frac{N}{N-1}}}}{F^{0}(x)^{\delta}}dx \\
        &\quad+\int_{B_{j}}\frac{k_{j}v_{j}f(x,k_{j}v_{j})}{F^{0}(x)^{\delta}}dx.
    \end{aligned}
    \end{equation}
    Using \eqref{(3.15)}, we have
    \begin{align}\label{(3.22)}
     \int_{\mathcal{W}_{l}}\frac{e^{\lambda_{0}|k_{j}v_{j}|^{\frac{N}{N-1}}}}{F^{0}(x)^{\delta}}
     dx
     &\geq \int_{\mathcal{W}_{l}}\frac{e^{\lambda_{N}(1-\frac{\delta}{N})v_{j}^{\frac{N}{N-1}}}}{F^{0}(x)^{\delta}}dx \nonumber \\
     &=\int_{\mathcal{W}_{\frac{l}{j}}}
     \frac{e^{\lambda_{N}(1-\frac{\delta}{N})v_{j}^{\frac{N}{N-1}}}}{F^{0}(x)^{\delta}}dx
     +\int_{\mathcal{W}_{l}\backslash \mathcal{W}_{\frac{l}{j}}}
     \frac{e^{\lambda_{N}(1-\frac{\delta}{N})v_{j}^{\frac{N}{N-1}}}}{F^{0}(x)^{\delta}}dx.
    \end{align}
    By \eqref{(3.18)}, we obtian
    \begin{align*}
        \int_{\mathcal{W}_{\frac{l}{j}}}\frac{e^{\lambda_{N}(1-\frac{\delta}{N})v_{j}^{\frac{N}{N-1}}}}{F^{0}(x)^{\delta}}
        dx
        &\geq e^{\lambda_{N}(1-\frac{\delta}{N})\left((N\kappa_{N})^{-\frac{1}{N-1}}\log j-(N\kappa_{N})^{-\frac{1}{N-1}}g(l)\frac{(N-2)!}{N^{N}}l^{N}+o_{j}(1)\right)}\int_{\mathcal{W}_{\frac{l}{j}}}\frac{1}{F^{0}(x)^{\delta}} dx\\
        &=\frac{N\kappa_{N}l^{N-\delta}}{N-\delta}e^{-(N-\delta)g(l)\frac{(N-2)!}{N^{N}}l^{N}+o_{j}(1)}.
    \end{align*}
    Moreover, by the change of variables
    $t=(\log j)^{-\frac{1}{N}}\|\tilde{v_{j}}\|_{S}^{-1}\log \frac{l}{k}$, we obtain
    \begin{align*}
        &\int_{\mathcal{W}_{l}\backslash \mathcal{W}_{\frac{l}{j}}}\frac{e^{\lambda_{N}(1-\frac{\delta}{N})v_{j}^{\frac{N}{N-1}}}}{F^{0}(x)^{\delta}}
     dx\\
    &=N\kappa_{N}\int_{\frac{l}{j}}^{l}k^{N-\delta-1}e^{(N-\delta)\left((\log j)^{-\frac{1}{N}}\|\tilde{v_{j}}\|_{S}^{-1}\log \frac{l}{k}\right)^{\frac{N}{N-1}}}dk\\
     &=N\kappa_{N}l^{N-\delta}\int_{0}^{(\log j)^{1-\frac{1}{N}}\|\tilde{v_{j}}\|_{S}^{-1}}(\log j)^{\frac{1}{N}}\|\tilde{v_{j}}\|_{S}e^{(N-\delta)t^{\frac{N}{N-1}}-(N-\delta)\|\tilde{v_{j}}\|_{S}(\log j)^{\frac{1}{N}}t}dt\\
     &\geq N\kappa_{N}l^{N-\delta}\int_{0}^{(\log j)^{1-\frac{1}{N}}\|\tilde{v_{j}}\|_{S}^{-1}}(\log j)^{\frac{1}{N}}\|\tilde{v_{j}}\|_{S}e^{-(N-\delta)\|\tilde{v_{j}}\|_{S}(\log j)^{\frac{1}{N}}t}dt\\
     &=\frac{N\kappa_{N}l^{N-\delta}}{N-\delta}\left(1-e^{-(N-\delta)\log j}\right).
    \end{align*}
    Letting $j\to\infty$ in \eqref{(3.22)}, we obtain
    \begin{equation}\label{(3.23)}
     \liminf\limits_{j\to\infty}
    \int_{\mathcal{W}_{l}}
     \frac{e^{\lambda_{0}|k_{j}v_{j}|^{\frac{N}{N-1}}}}{F^{0}(x)^{\delta}}dx
     \geq \frac{N\kappa_{N}l^{N-\delta}}{N-\delta}\left(1+e^{-(N-\delta)g(l)\frac{(N-2)!}{N^{N}}l^{N}}\right).
    \end{equation}

    Since $v_{j}\to 0$ a.e. in $\mathbb{R}^{N}$, by Lebesgue's dominated convergence theorem, it follows that
    \begin{equation}\label{(3.24)}
        \lim\limits_{j\to\infty}\int_{B_{j}}\frac{e^{\lambda_{0}|k_{j}v_{j}|^{\frac{N}{N-1}}}}{F^{0}(x)^{\delta}}dx=\int_{\mathcal{W}_{l}}\frac{1}{F^{0}(x)^{\delta}}dx=\frac{N\kappa_{N}l^{N-\delta}}{N-\delta},
    \end{equation}
    \begin{equation}\label{(3.25)}
        \lim\limits_{j\to\infty}\int_{B_{j}}\frac{k_{j}v_{j}f(x,k_{j}v_{j})}{F^{0}(x)^{\delta}}dx=0.
    \end{equation}
    Combining \eqref{(3.20)}--\eqref{(3.25)}, we obtain
    $$\left(\frac{N-\delta}{N}\frac{\lambda_{N}}{\lambda_{0}}\right)^{N-1}\geq (\beta_{0}-\sigma)\frac{N\kappa_{N}l^{N-\delta}}{N-\delta}e^{-(N-\delta)g(l)\frac{(N-2)!}{N^{N}}l^{N}}.$$
    By the definition of $\lambda_{N}$ and the arbitrariness of $\sigma>0$, we have
    $$\beta_{0}\leq \frac{(N-\delta)^{N}}{\lambda_{0}^{N-1}l^{N-\delta}}e^{(N-\delta)g(l)\frac{(N-2)!}{N^{N}}l^{N}},$$
    which contradicts \eqref{(3.13)}. Therefore, \eqref{(3.11)} is proved.

  In addition, \eqref{(3.12)} follows from \eqref{(3.11)} and the definition of $E_{\epsilon}$. Then we complete the proof of Lemma \ref{Le3.5}.
     \end{proof}

   Next, we research some properties of Palais-Smale sequences of the functionals $E$ and $E_{\epsilon}$.

     \begin{lemma}\label{Le3.6}
        Suppose that $(A_{1})$, $(A_{2})$, $(B_{1})$, $(B_{2})$ and $(B_{3})$ hold. Let $(u_{j})\subset S$ be a Palais-Smale sequence of $E_{\epsilon}$, i.e.
        \begin{equation}\label{(3.26)}
            E_{\epsilon}(u_{j})\to c, \quad E_{\epsilon}^{'}(u_{j})\to 0 \quad \text{in}\ S^{*}\ \text{as}\ j\to \infty,
        \end{equation}
        where $S^{*}$ is the dual space of $S$. Then there exist a subsequence, still denoted by $(u_{j})$, and a function $u\in S$ such that $u_{j}\rightharpoonup u$ weakly in $S$, $u_{j}\to u$ strongly in $L^{q}(\mathbb{R}^{N})$ for every $q\geq 1$, and
        \begin{equation*}
            \left\{
            \begin{aligned}
                &\frac{f(x,u_{j})}{F^{0}(x)^{\delta}}\to \frac{f(x,u)}{F^{0}(x)^{\delta}} &&\text{strongly in}\ L^{1}(\mathbb{R}^{N}),\\
                &\frac{G(x,u_{j})}{F^{0}(x)^{\delta}}\to \displaystyle\frac{G(x,u)}{F^{0}(x)^{\delta}}
                &&\text{strongly in}\ L^{1}(\mathbb{R}^{N}),\\
                &\nabla u_{j}(x)\to \nabla u(x)  &&a.e.\ \text{in}\ \mathbb{R}^{N},\\
                &F^{N-1}(\nabla u_{j})F_{\xi}(\nabla u_{j})\rightharpoonup F^{N-1}(\nabla u)F_{\xi}(\nabla u)
                &&\text{weakly in}\ (L^{\frac{N}{N-1}}(\mathbb{R}^{N}))^{N}.
            \end{aligned}
            \right.
        \end{equation*}
        Moreover, $u$ is a weak solution of \eqref{(1.6)}; the case $\epsilon=0$ is analogous.
    \end{lemma}

    \begin{proof}
        Let $(u_{j})$ be a Palais-Smale sequence of $E_{\epsilon}$. We have from \eqref{(3.2)} and \eqref{(3.4)} that
        \begin{equation}\label{(3.27)}
            E_{\epsilon}(u_{j})=\frac{1}{N}\|u_{j}\|_{S}^{N}-\int_{\mathbb{R}^{N}}\frac{G(x,u_{j})}{F^{0}(x)^{\delta}}dx-\epsilon\int_{\mathbb{R}^{N}}hu_{j}dx\to c \quad \text{as}\ j\to \infty,
        \end{equation}
        \begin{align}\label{(3.28)}
            &\left|\langle E^{'}_{\epsilon}(u_{j}), \varphi \rangle\right|\nonumber\\
            &=\left|\int_{\mathbb{R}^{N}}(F^{N-1}(\nabla u_{j})F_{\xi}(\nabla u_{j})\nabla \varphi+V(x)|u_{j}|^{N-2}u_{j}\varphi)dx
           -\int_{\mathbb{R}^{N}}\frac{f(x,u_{j})}{F^{0}(x)^{\delta}}\varphi dx
           -\epsilon\int_{\mathbb{R}^{N}}h\varphi dx\right|\nonumber\\
           &\leq\rho_{j}\|\varphi\|_{S}
        \end{align}
        for every $\varphi\in S$, where $\rho_{j}\to 0$ as $j\to \infty$.
        By $f(x,t)\equiv 0$ for $t<0$ and $(B_{2})$, we obtain that $0\leq \eta G(x,u_{j})\leq u_{j}f(x,u_{j})$ for some $\eta>N$. Taking $\varphi=u_{j}$ in \eqref{(3.28)} and multiplying \eqref{(3.27)} by $\eta$, we have
        \begin{align*}
            \left(\frac{\eta}{N}-1\right)\|u_{j}\|_{S}^{N}
            &\leq \left(\frac{\eta}{N}-1\right)\|u_{j}\|_{S}^{N}-\int_{\mathbb{R}^{N}}\frac{\eta G(x,u_{j})-f(x,u_{j})u_{j}}{F^{0}(x)^{\delta}}dx\\
            &\leq \eta|c|+\rho_{j}\|u_{j}\|_{S}+(\eta +1)\epsilon\|h\|_{S^{*}}\|u_{j}\|_{S}.
        \end{align*}
        Therefore $\|u_{j}\|_{S}$ is bounded. Combining \eqref{(3.27)} and \eqref{(3.28)}, we deduce
         \begin{equation}\label{(3.29)}
          \int_{\mathbb{R}^{N}}
            \frac{f(x,u_{j})u_{j}}{F^{0}(x)^{\delta}}dx\leq C, \quad
         \int_{\mathbb{R}^{N}}
            \frac{G(x,u_{j})}{F^{0}(x)^{\delta}}dx\leq C
        \end{equation}
        for some constant $C$ depending on $\eta$, $N$ and $\|h\|_{S^{*}}$.
        Since $(u_{j})$ is bounded in $S$, up to a subsequence, we may assume that $u_{j}\rightharpoonup u$ weakly in $S$ for some $u\in S$.
        By Lemma \ref{Le2.5}, $u_{j}\to u$ strongly in $L^{q}(\mathbb{R}^{N})$ for every $q\geq 1$, and consequently $u_{j}\to u$ a.e. in $\mathbb{R}^{N}$.

        Next, we will show that, up to a subsequence,
        \begin{equation}\label{(3.30)}
          \lim\limits_{j\to\infty}
          \int_{\mathbb{R}^{N}}\frac{\left|f(x,u_{j})-f(x,u)\right|}{F^{0}(x)^{\delta}}dx=0.
        \end{equation}
         Since $f(x,\cdot)\geq 0$, it suffices to prove that, up to a subsequence,
        \begin{equation}\label{(3.31)}
          \lim\limits_{j\to\infty}
          \int_{\mathbb{R}^{N}}\frac{f(x,u_{j})}{F^{0}(x)^{\delta}}dx
          =\int_{\mathbb{R}^{N}}\frac{f(x,u)}{F^{0}(x)^{\delta}}dx.
        \end{equation}
         Since $u$ and $\frac{f(x,u)}{F^{0}(x)^{\delta}}$ belong to $L^{1}(\mathbb{R}^{N})$, we obtain
        \begin{equation}\label{(3.32)}
          \lim\limits_{\sigma\to+\infty}\int_{|u|\geq \sigma}\frac{f(x,u)}{F^{0}(x)^{\delta}}dx=0.
        \end{equation}
         For every $\gamma>0$, there exists some $K>\frac{C}{\gamma}$ such that
        \begin{equation}\label{(3.33)}
            \int_{|u|\geq K}\frac{f(x,u)}{F^{0}(x)^{\delta}}dx<\gamma,
        \end{equation}
        where $C$ is the constant in \eqref{(3.29)}.
       Then by \eqref{(3.29)}, it follows that
        \begin{equation}\label{(3.34)}
            \int_{|u_{j}|\geq K}\frac{f(x,u_{j})}{F^{0}(x)^{\delta}}dx
            \leq \frac{1}{K}\int_{|u_{j}|\geq K}\frac{f(x,u_{j})u_{j}}{F^{0}(x)^{\delta}}dx<\gamma.
        \end{equation}

        Now we claim that
        \begin{equation}\label{(3.35)}
            \lim\limits_{j\to\infty}\int_{|u_{j}|< K}\frac{f(x,u_{j})}{F^{0}(x)^{\delta}}dx=\int_{|u|< K}\frac{f(x,u)}{F^{0}(x)^{\delta}}dx.
        \end{equation}
        Indeed, it is easy to see that
        \begin{align*}
            &\left|\int_{|u_{j}|< K}\frac{f(x,u_{j})}{F^{0}(x)^{\delta}}dx-\int_{|u|< K}\frac{f(x,u)}{F^{0}(x)^{\delta}}dx\right|\\
            &\leq\left|\int_{\mathbb{R}^{N}}\frac{\mathcal{X}_{|u_{j}|<K}(|f(x,u_{j})|-|f(x,u)|)}{F^{0}(x)^{\delta}}dx\right|
            +\left|\int_{\mathbb{R}^{N}}\frac{(\mathcal{X}_{|u_{j}|<K}-\mathcal{X}_{|u|<K})|f(x,u)|}{F^{0}(x)^{\delta}}dx\right|.
        \end{align*}
        From $(B_{1})$, there exists a constant $C_{1}$, depending only on $K$, such that $|f(x,u_{j}(x))|\leq C_{1}|u_{j}(x)|^{N-1}$ for every $x\in \{x\in\mathbb{R}^{N}: |u_{j}|<K\}$.  Notice that $\frac{|u_{j}|^{N-1}}{F^{0}(x)^{\delta}}\to \frac{|u|^{N-1}}{F^{0}(x)^{\delta}}$ strongly in $L^{1}(\mathbb{R}^{N})$ and $u_{j}\to u$ a.e. in $\mathbb{R}^{N}$. Moreover, $\mathcal{X}_{|u_{j}|<K}(|f(x,u_{j})|-|f(x,u)|)\to 0$ a.e. in $\mathbb{R}^{N}$. By the generalized Lebesgue dominated convergence theorem, we obtain
        \begin{equation*}
        \left|\int_{\mathbb{R}^{N}}\frac{\mathcal{X}_{|u_{j}|<K}(|f(x,u_{j})|-|f(x,u)|)}{F^{0}(x)^{\delta}}dx\right|\to 0 \quad \text{as}\ j\to\infty.
        \end{equation*}
        Since
        \begin{equation*}
            \{x\in\mathbb{R}^{N}:|u_{j}(x)|<K\}\backslash \{x\in\mathbb{R}^{N}:|u(x)|<K\}\subset \{x\in\mathbb{R}^{N}:|u(x)|\geq K\},
        \end{equation*}
       then we have
        \begin{equation*}
            \left|\int_{\mathbb{R}^{N}}\frac{(\mathcal{X}_{|u_{j}|<K}-\mathcal{X}_{|u|<K})|f(x,u)|}{F^{0}(x)^{\delta}}dx\right|\leq \int_{|u|\geq K}\frac{f(x,u)}{F^{0}(x)^{\delta}}dx<\gamma.
        \end{equation*}
        With the help of \eqref{(3.33)}, \eqref{(3.34)} and \eqref{(3.35)}, we can find $L>0$ such that, for every $j>L$,
        \begin{equation*}
            \left|\int_{\mathbb{R}^{N}}\frac{f(x,u_{j})}{F^{0}(x)^{\delta}}dx-\int_{\mathbb{R}^{N}}\frac{f(x,u)}{F^{0}(x)^{\delta}}dx\right|<3\gamma.
        \end{equation*}
        Thus \eqref{(3.31)} holds, which proves \eqref{(3.30)}.

        Moreover, by $(B_{1})$ and $(B_{3})$, there exist constants $c_{1}, c_{2}>0$ such that
        \begin{equation*}
            G(x,u_{j})\leq c_{1}|u_{j}|^{N}+c_{2}f(x,u_{j}).
        \end{equation*}
        Using \eqref{(3.30)}, Lemma \ref{Le2.5} and the generalized Lebesgue dominated convergence theorem, we obtain
        \begin{equation*}
        \lim\limits_{j\to\infty}\int_{\mathbb{R}^{N}}
         \frac{\left|G(x,u_{j})-G(x,u)\right|}{F^{0}(x)^{\delta}}dx=0.
        \end{equation*}

        To prove the remaining part of the lemma, we define, for every given $\ell>0$, the concentration set
        \begin{equation*}
         \Sigma_{\ell}=\left\{x\in\mathbb{R}^{N}: \lim\limits_{l\to 0}\lim\limits_{j\to\infty}\int_{\mathcal{W}_{l}(x)}\left(F^{N}(\nabla u_{j})+V(x)|u_{j}|^N\right)dx\geq \ell\right\}.
        \end{equation*}
        Since $(u_{j})$ is bounded in $S$, the set $\Sigma_{\ell}$ is finite. For every $x^{0}\in \mathbb{R}^{N}\backslash \Sigma_{\ell}$, there exists $l$ with $0<l<\mathrm{dist}(x^{0},\Sigma_{\ell})$ such that
        \begin{equation*}
           \lim\limits_{j\to\infty}
           \int_{\mathcal{W}_{l}(x^{0})}\left(F^{N}(\nabla u_{j})+V(x)|u_{j}|^N\right)dx<\ell.
        \end{equation*}
        Hence, for sufficiently large $j$,
        \begin{equation}\label{(3.36)}
           \int_{\mathcal{W}_{l}(x^{0})}\left(F^{N}(\nabla u_{j})+V(x)|u_{j}|^N\right)dx<\ell.
        \end{equation}
        Then by condition $(A_1)$, we have
        \begin{equation}\label{(3.36b)}
           \int_{\mathcal{W}_{l}(x^{0})}\left(F^{N}(\nabla u_{j})+V^{*}|u_{j}|^N\right)dx\leq\int_{\mathcal{W}_{l}(x^{0})}\left(F^{N}(\nabla u_{j})+V(x)|u_{j}|^N\right)dx<\ell.
        \end{equation}

        By Lemma \ref{Le2.1}, for sufficiently small $\ell>0$, there exists $p>1$ such that
        \begin{equation}\label{(3.37)}
            \int_{\mathcal{W}_{l}(x^{0})}\frac{|f(x,u_{j})|^{p}}{F^{0}(x)^{\delta}}dx\leq C.
        \end{equation}
        For any $H>0$, we suppose that
        \begin{equation*}
            D_{H}=\{x\in\mathcal{W}_{l}(x^{0}): |u(x)|\geq H\}.
        \end{equation*}
        Then
        \begin{align*}
            \int_{D_{H}}\frac{|f(x,u_{j})-f(x,u)||u|}{F^{0}(x)^{\delta}}dx
            &\leq \left(\int_{D_{H}}\frac{|f(x,u_{j})-f(x,u)|^{q}}{F^{0}(x)^{\delta}}dx\right)^{\frac{1}{q}}\left(\int_{D_{H}}\frac{|u|^{q^{0}}}{F^{0}(x)^{\delta}}dx\right)^{\frac{1}{q^{0}}}\\
            &\leq \left(\left(\int_{\mathcal{W}_{l}(x^{0})}\frac{|f(x,u_{j})|^{q}}{{F^{0}(x)^{\delta}}}dx\right)^{\frac{1}{q}}+\left(\int_{\mathcal{W}_{l}(x^{0})}\frac{|f(x,u)|^{q}}{{F^{0}(x)^{\delta}}}dx\right)^{\frac{1}{q}}\right)\\
           &\times\left(\int_{\mathcal{W}_{l}(x^{0})}\frac{1}{F^{0}(x)^{\delta k}}dx\right)^{\frac{1}{q^{0}k}}\left(\int_{D_{H}}|u|^{q^{0}k^{0}}dx\right)^{\frac{1}{q^{0}k^{0}}}\\
           &\leq C\left(\int_{D_{H}}|u|^{q^{0}k^{0}}dx\right)^{\frac{1}{q^{0}k^{0}}},
        \end{align*}
        where $\frac{1}{q}+\frac{1}{q^{0}}=1, \frac{1}{k}+\frac{1}{k^{0}}=1$ and $0<\delta k<N$.
        Since $u\in L^{q^{0}k^{0}}(\mathcal{W}_{l}(x^{0}))$, by choosing $H$ sufficiently large and for any $d>0$, we have
        \begin{equation}\label{(3.38)}
            \int_{D_{H}}\frac{|f(x,u_{j})-f(x,u)||u|}{F^{0}(x)^{\delta}}dx<d.
        \end{equation}
         By \eqref{(3.30)}, we obtain
        \begin{equation}\label{(3.39)}
            \lim\limits_{j\to \infty}\int_{\mathcal{W}_{l}(x^{0})\backslash D_{H}}\frac{|f(x,u_{j})-f(x,u)||u|}{F^{0}(x)^{\delta}}dx=0.
        \end{equation}
        Thus
        \begin{equation*}
            \lim\limits_{j\to \infty}\int_{\mathcal{W}_{l}(x^{0})}\frac{|f(x,u_{j})-f(x,u)||u|}{F^{0}(x)^{\delta}}dx<d.
        \end{equation*}
        Since $d>0$ is arbitrary, we obtain
        \begin{equation}\label{(3.40)}
            \lim\limits_{j\to \infty}\int_{\mathcal{W}_{l}(x^{0})}\frac{|f(x,u_{j})-f(x,u)||u|}{F^{0}(x)^{\delta}}dx=0.
        \end{equation}
        By the H\"{o}lder inequality, \eqref{(3.37)} and $u_{j}\to u$ strongly in $L^{q}(\mathbb{R}^{N})$ for every $q\geq 1$, we obtain
        \begin{align}\label{(3.41)}
            &\int_{\mathcal{W}_{l}(x^{0})}\frac{|f(x,u_{j})||u_{j}-u|}{F^{0}(x)^{\delta}}dx\nonumber\\
            &\leq \left(\int_{\mathcal{W}_{l}(x^{0})}\frac{|f(x,u_{j})|^{q_{1}}}{F^{0}(x)^{\delta}}dx\right)^{\frac{1}{q_{1}}} \left(\int_{\mathcal{W}_{l}(x^{0})}\frac{1}{F^{0}(x)^{\delta k_{1}}}dx\right)^{\frac{1}{q_{2}k_{1}}}
            \left(\int_{\mathcal{W}_{l}(x^{0})}|u_{j}-u|^{q_{2}k_{2}}dx\right)^{\frac{1}{q_{2}k_{2}}}\nonumber\\
            &\leq C\left(\int_{\mathcal{W}_{l}(x^{0})}|u_{j}-u|^{q_{2}k_{2}}dx\right)^{\frac{1}{q_{2}k_{2}}} \to 0 \quad \text{as}\ j\to\infty,
            \end{align}
        where $\frac{1}{q_{1}}+\frac{1}{q_{2}}=1, \frac{1}{k_{1}}+\frac{1}{k_{2}}=1$ and $0<\delta k_{1}<N$.
       Then by \eqref{(3.40)} and \eqref{(3.41)}, we have
        \begin{equation*}
           \lim\limits_{j\to\infty}
           \int_{\mathcal{W}_{l}(x^{0})}\frac{|f(x,u_{j})u_{j}-f(x,u)u|}{F^{0}(x)^{\delta}}dx=0.
        \end{equation*}
        For any compact set $I\subset\subset\mathbb{R}^{N}\backslash\Sigma_{\ell}$, we have
        \begin{equation}\label{(3.42)}
            \lim\limits_{j\to\infty}
           \int_{I}\frac{|f(x,u_{j})u_{j}-f(x,u)u|}{F^{0}(x)^{\delta}}dx=0.
        \end{equation}
        Next we claim that
        \begin{equation}\label{(3.43)}
            \lim\limits_{j\to\infty}
             \int_{I}F^{N}(\nabla u_{j}-\nabla u)dx=0
        \end{equation}
        for any compact set $I\subset\subset\mathbb{R}^{N}\backslash\Sigma_{\ell}$.
        In fact, it suffices to prove that, for every $x^{0}\in \mathbb{R}^{N}\backslash\Sigma_{\ell}$ and $l$,
        \begin{equation}\label{(3.44)}
             \lim\limits_{j\to\infty}
             \int_{\mathcal{W}_{\frac{l}{2}}(x^{0})}F^{N}(\nabla u_{j}-\nabla u)dx=0.
        \end{equation}
        To this end, take $\psi\in C_{0}^{\infty}(\mathcal{W}_{l}(x^{0}))$ with $0\leq \psi \leq 1$ and $\psi\equiv 1$ on $\mathcal{W}_{\frac{l}{2}}(x^{0})$. Clearly, $(\psi u_{j})$ is bounded in $S$. Taking $\varphi=\psi u_{j}$ and $\varphi=\psi u$ in \eqref{(3.28)}, respectively, it follows that
        \begin{align}\label{(3.45)}
           &\int_{\mathcal{W}_{l}(x^{0})}
           \psi(F^{N-1}(\nabla u_{j})F_{\xi}(\nabla u_{j})-F^{N-1}(\nabla u)F_{\xi}(\nabla u))(\nabla u_{j}-\nabla u)dx\nonumber\\
           &\leq \int_{\mathcal{W}_{l}(x^{0})}
           F^{N-1}(\nabla u_{j})F_{\xi}(\nabla u_{j})\nabla\psi(u-u_{j})dx+
           \int_{\mathcal{W}_{l}(x^{0})}
           \psi F^{N-1}(\nabla u)F_{\xi}(\nabla u)(\nabla u-\nabla u_{j})dx\nonumber\\
           &\quad+\int_{\mathcal{W}_{l}(x^{0})}
           \psi(u_{j}-u)\frac{f(x,u_{j})}{F^{0}(x)^{\delta}}dx
           +\int_{\mathcal{W}_{l}(x^{0})}
           \psi V(x)|u_{j}|^{N-2}u_{j}(u-u_{j})dx\nonumber\\
           &\quad+\rho_{j}\|\psi u_{j}\|_{S}+\rho_{j}\|\psi u\|_{S}+\epsilon\|h\|_{S^{*}}\|u_{j}-u\|_{S}.
        \end{align}

         By the H\"{o}lder inequality and the compact embedding $S\hookrightarrow\hookrightarrow L^{N}(\mathbb{R}^{N})$, we obtain
        \begin{equation}\label{(3.46)}
            \lim\limits_{j\to\infty}
            \int_{\mathcal{W}_{l}(x^{0})}
           F^{N-1}(\nabla u_{j})F_{\xi}(\nabla u_{j})\nabla\psi(u-u_{j})dx=0.
        \end{equation}
        From $\nabla u_{j}\rightharpoonup \nabla u$ weakly in $(L^{N}(\mathbb{R}^{N}))^{N}$, we have
        \begin{equation}\label{(3.47)}
            \lim\limits_{j\to\infty}
            \int_{\mathcal{W}_{l}(x^{0})}
           \psi F^{N-1}(\nabla u)F_{\xi}(\nabla u)(\nabla u-\nabla u_{j})dx=0.
        \end{equation}
        By the H\"{o}lder inequality and \eqref{(3.41)}, we obtain
        \begin{equation}\label{(3.48)}
             \lim\limits_{j\to\infty}
             \int_{\mathcal{W}_{l}(x^{0})}
           \psi(u_{j}-u)\frac{f(x,u_{j})}{F^{0}(x)^{\delta}}dx=0.
        \end{equation}
        Applying the H\"{o}lder inequality and the compact embedding $S\hookrightarrow\hookrightarrow L^{N}(\mathbb{R}^{N})$, we obtain
        \begin{equation}\label{(3.49)}
            \lim\limits_{j\to\infty}
            \int_{\mathcal{W}_{l}(x^{0})}
           \psi V(x)|u_{j}|^{N-2}u_{j}(u-u_{j})dx=0.
        \end{equation}
        Combining \eqref{(3.46)}, \eqref{(3.47)}, \eqref{(3.48)}, \eqref{(3.49)} and $\rho_{j}\to 0$, we conclude that the left-hand side of \eqref{(3.45)} tends to $0$ as $j\to\infty$.
        Inequality \eqref{(3.44)} follows from \eqref{(3.45)} and \eqref{(2.5)}. A standard covering argument then yields \eqref{(3.43)}. Since $\Sigma_{\ell}$ is finite, we get
        \begin{equation*}
            \nabla u_{j}(x)\to \nabla u(x) \quad a.e.\ \text{in}\ \mathbb{R}^{N}.
        \end{equation*}
        Therefore, up to a subsequence, we have
        \begin{equation*}
            F^{N-1}(\nabla u_{j})F_{\xi}(\nabla u_{j})\rightharpoonup F^{N-1}(\nabla u)F_{\xi}(\nabla u) \quad \text{weakly in}\ (L^{\frac{N}{N-1}}(\mathcal{W}_{r}(0)))^{N}
        \end{equation*}
        for every $r>0$. For each $\varphi\in C_{0}^{\infty}(\mathbb{R}^{N})$, there exists $r_{0}>0$ such that $\mathrm{supp}\,\varphi\subset\mathcal{W}_{r_{0}}(0)$.
        By the density of $C_{0}^{\infty}(\mathbb{R}^{N})$ in $L^{\frac{N}{N-1}}(\mathbb{R}^{N})$, for every $\varphi\in L^{\frac{N}{N-1}}(\mathbb{R}^{N})$, we have
        \begin{equation*}
            \lim\limits_{j\to\infty}
            \int_{\mathbb{R}^{N}}
            (F^{N-1}(\nabla u_{j})F_{\xi}(\nabla u_{j})-F^{N-1}(\nabla u)F_{\xi}(\nabla u))\varphi dx=0,
        \end{equation*}
        which implies that
        \begin{equation}\label{(3.50)}
            F^{N-1}(\nabla u_{j})F_{\xi}(\nabla u_{j})\rightharpoonup F^{N-1}(\nabla u)F_{\xi}(\nabla u) \quad \text{weakly in}\ (L^{\frac{N}{N-1}}(\mathbb{R}^{N}))^{N}.
        \end{equation}
        Letting $j\to\infty$ in \eqref{(3.28)}, it follows that
        \begin{equation*}
            \int_{\mathbb{R}^{N}}(F^{N-1}(\nabla u)F_{\xi}(\nabla u)\nabla \phi+V(x)|u|^{N-2}u\varphi)dx
           -\int_{\mathbb{R}^{N}}\frac{f(x,u_{j})}{F^{0}(x)^{\delta}}\varphi dx
           -\epsilon\int_{\mathbb{R}^{N}}h\varphi dx=0
        \end{equation*}
        for every $\varphi\in C_{0}^{\infty}(\mathbb{R}^{N})$, which is dense in $S$. Thus $u$ is a weak solution of \eqref{(1.6)}; the case $\epsilon=0$ follows in the same way. The proof of Lemma \ref{Le3.6} is finished.
    \end{proof}

    \section{Proof of the main results}

    \subsection{Proof of Theorem \ref{Th1.1}}

    By Lemmas \ref{Le3.1} and \ref{Le3.2}, $E$ satisfies all the assumptions of the mountain-pass theorem without the Palais-Smale condition \cite{42}: $E\in\mathcal{C}^{1}(S, \mathbb{R})$; $E(0)=0$; $E(u)\geq \gamma>0$ when $\|u\|_{S}=l$; and $E(e)<0$ for some $e\in S$ with $\|e\|_{S}>l$. Hence there exists a sequence $(u_{j})$ in $S$ such that
    \begin{equation*}
            E(u_{j})\to c>0, \quad E^{'}(u_{j})\to 0 \quad \text{in}\ S^{*}\ \text{as}\ j\to \infty,
    \end{equation*}
    where $$c=\min\limits_{\mu\in\Gamma} \max\limits_{u\in\mu}E(u)\geq \gamma$$
    is the min-max value of $E$, where $\Gamma=\{\mu\in\mathcal{C}([0,1],S):\mu(0)=0, \mu(1)=e\}$.
    By \eqref{(3.1)} and \eqref{(3.3)}, we obtain
    \begin{equation}\label{(4.1)}
            E(u_{j})=\frac{1}{N}\|u_{j}\|_{S}^{N}-\int_{\mathbb{R}^{N}}\frac{G(x,u_{j})}{F^{0}(x)^{\delta}}dx\to c \quad \text{as}\ j\to \infty,
        \end{equation}
        \begin{align}\label{(4.2)}
            \left|\langle E^{'}(u_{j}), \varphi \rangle\right|
            &=\left|\int_{\mathbb{R}^{N}}(F^{N-1}(\nabla u_{j})F_{\xi}(\nabla u_{j})\nabla \varphi+V(x)|u_{j}|^{N-2}u_{j}\varphi)dx
           -\int_{\mathbb{R}^{N}}\frac{f(x,u_{j})}{F^{0}(x)^{\delta}}\varphi dx
           \right|\nonumber\\
           &\leq\rho_{j}\|\varphi\|_{S}
        \end{align}
    for every $\varphi\in S$, where $\rho_{j}\to 0$ as $j\to \infty$.
    By Lemma \ref{Le3.6}, there exists a subsequence, still denoted by $(u_{j})$, such that
    \begin{equation}\label{(4.3)}
        \left\{
        \begin{aligned}
            &u_{j}\rightharpoonup u\quad\text{weakly in}\ S,\\
            &u_{j}\to u\quad\text{strongly in}\ L^{q}(\mathbb{R}^{N}), \forall q\geq 1,\\
            &\lim\limits_{j\to\infty}
          \int_{\mathbb{R}^{N}}\frac{G(x,u_{j})}{F^{0}(x)^{\delta}}dx
          =\int_{\mathbb{R}^{N}}\frac{G(x,u)}{F^{0}(x)^{\delta}}dx,\\
          &u\ \text{is a weak solution of \eqref{(1.5)}}.
        \end{aligned}
        \right.
    \end{equation}
    We now show that $u$ is nontrivial. Suppose that $u\equiv 0$ by the contrary. Note that $G(x,0)=0$ for every $x\in\mathbb{R}^{N}$. Combining \eqref{(4.1)} and \eqref{(4.3)}, we have
    \begin{equation}\label{(4.4)}
        \lim\limits_{j\to\infty}\|u_{j}\|_{S}^{N}=Nc>0.
    \end{equation}
    By $(B_{5})$ and Lemma \ref{Le3.5}, we have $0<c<\frac{1}{N}\left(\frac{N-\delta}{N}\frac{\lambda_{N}}{\lambda_{0}}\right)^{N-1}$. Thus we can find $\sigma_{0}>0$ and $L>0$ such that $\|u_{j}\|_{S}^{N}\leq \left(\frac{N-\delta}{N}\frac{\lambda_{N}}{\lambda_{0}}-\sigma_{0}\right)^{N-1}$ for every $j>L$. Choosing $q>1$ sufficiently close to $1$, we may assume that $q\lambda_{0}\|u_{j}\|_{S}^{\frac{N}{N-1}}\leq \left(1-\frac{\delta}{N}\right)\lambda_{N}-\frac{\lambda_{0}\sigma_{0}}{2}$ for every $j>L$.
    By $(B_{1})$, it follows that
    \begin{equation*}
        \left|f(x,u_{j})u_{j}\right|\leq a_{1}|u_{j}|^{N}+a_{2}|u_{j}|\Phi\left(N,\lambda_{0}|u_{j}|^{\frac{N}{N-1}}\right).
    \end{equation*}
    By the H\"{o}lder inequality, Lemmas \ref{Le2.4} and \ref{Le2.2}, we obtain
    \begin{align*}
       \int_{\mathbb{R}^{N}}
       \frac{\left|f(x,u_{j})u_{j}\right|}{F^{0}(x)^{\delta}}dx
       &\leq a_{1}\int_{\mathbb{R}^{N}}\frac{|u_{j}|^{N}}{F^{0}(x)^{\delta}}dx
       +a_{2}\int_{\mathbb{R}^{N}}
       \frac{|u_{j}|\Phi(N,\lambda_{0}|u_{j}|^{\frac{N}{N-1}})}{F^{0}(x)^{\delta}}dx\\
       &\leq a_{1}\int_{\mathbb{R}^{N}}\frac{|u_{j}|^{N}}{F^{0}(x)^{\delta}}dx
       +a_{2}\left(\int_{\mathbb{R}^{N}}
       \frac{|u_{j}|^{q^{0}}}{F^{0}(x)^{\delta}}dx\right)^{\frac{1}{q^{0}}}
       \left(\int_{\mathbb{R}^{N}}
       \frac{\Phi(N,q\lambda_{0}|u_{j}|^{\frac{N}{N-1}})}{F^{0}(x)^{\delta}}dx\right)^{\frac{1}{q}}\\
       &\leq a_{1}\int_{\mathbb{R}^{N}}\frac{|u_{j}|^{N}}{F^{0}(x)^{\delta}}dx
       +C\left(\int_{\mathbb{R}^{N}}
       \frac{|u_{j}|^{q^{0}}}{F^{0}(x)^{\delta}}dx\right)^{\frac{1}{q^{0}}}
       \to 0 \quad\text{as}\ j\to\infty,
    \end{align*}
    where $\frac{1}{q^{0}}+\frac{1}{q}=1$.
    Taking $\varphi=u_{j}$ in \eqref{(4.2)}, we conclude that $$\|u_{j}\|_{S}^{N}\to 0 \quad \text{as}\ j\to\infty,$$ which contradicts \eqref{(4.4)}, thus $u\not\equiv 0$. Then the proof of Theorem \ref{Th1.1} is completed. \qed

    \subsection{Proof of Theorem \ref{Th1.2}}

    \textbf{Step 1: Existence of a positive mountain-pass type solution.}

    By Lemmas \ref{Le3.1} and \ref{Le3.2}, there exists $\epsilon_{1}>0$ such that, when $0<\epsilon<\epsilon_{1}$, $E_{\epsilon}$ fulfills the following conditions of the mountain-pass theorem without the Palais-Smale condition: $E_{\epsilon}\in\mathcal{C}^{1}(S, \mathbb{R})$; $E_{\epsilon}(0)=0$; $E_{\epsilon}(u)\geq \theta_{\epsilon}>0$ when $\|u\|_{S}=l_{\epsilon}$; and $E_{\epsilon}(e)<0$ for some $e\in S$ with $\|e\|_{S}>\max\{l_{\epsilon}, 1\}$.
    Then there exists a sequence $(w_{j})$ in $S$ such that
    \begin{equation}\label{(4.5)}
            E_{\epsilon}(w_{j})\to c_{K}>0, \quad E_{\epsilon}^{'}(w_{j})\to 0 \quad \text{in}\ S^{*}\ \text{as}\ j\to \infty,
    \end{equation}
    where $$c_{K}=\min\limits_{\mu\in\Gamma} \max\limits_{u\in\mu}E_{\epsilon}(u)\geq \theta_{\epsilon}$$
    is the min-max value of $E_{\epsilon}$ and $\Gamma=\{\mu\in\mathcal{C}([0,1],S):\mu(0)=0, \mu(1)=e\}$.
    By Lemma \ref{Le3.6}, up to a subsequence, $w_{j}\rightharpoonup u_{K}$ weakly in $S$ and $u_{K}$ is a weak solution of \eqref{(1.6)}. Since $h\geq 0$ and $f(x,t)\equiv 0$ for $t<0$, the negative-part argument in the introduction shows that $u_{K}$ is a positive weak solution of \eqref{(1.6)}.

    \textbf{Step 2: Existence of a positive solution with negative energy.}

    By Lemma \ref{Le3.2},
    we have $l_{\epsilon}\to 0$ as $\epsilon\to 0$. Thus there exists $\epsilon_{2}$ with $0<\epsilon_{2}<\epsilon_{1}$ such that, for $0<\epsilon<\epsilon_{2}$, we have
    \begin{equation}\label{(4.6)}
        l_{\epsilon}<\left(\frac{N-\delta}{N}\frac{\lambda_{N}}{\lambda_{0}}\right)^{\frac{N-1}{N}}.
    \end{equation}
    By $(B_{1})$ and $(B_{2})$, we have
    \begin{equation}\label{(4.7)}
        G(x,u)\leq a_{1}|u|^{N}+a_{2}|u|\Phi\left(N,\lambda_{0}\|u\|_{S}^{\frac{N}{N-1}}\left(\frac{|u|}{\|u\|_{S}}\right)^{\frac{N}{N-1}}\right).
    \end{equation}
    When $\|u\|_{S}\leq l_{\epsilon}$, we have $\lambda_{0}\|u\|_{S}^{\frac{N}{N-1}}<\left(1-\frac{\delta}{N}\right)\lambda_{N}$, and therefore, by Lemmas \ref{Le2.4} and \ref{Le2.2}, $\frac{G(x,u)}{F^{0}(x)^{\delta}}$ is bounded in $L^{p}(\mathbb{R}^{N})\cap L^{1}(\mathbb{R}^{N})$ for some $p>1$.
    Thus $E_{\epsilon}$ is bounded from below on the closed Wulff ball $\overline{\mathcal{W}}_{l_{\epsilon}}=\{u\in S: \|u\|_{S}\leq l_{\epsilon}\}$.

    The set $\overline{\mathcal{W}}_{l_{\epsilon}}\subset S$ is a complete metric space with the metric induced by the norm of $S$; moreover, it is convex, and $E_{\epsilon}$ is of class $\mathcal{C}^{1}$ and bounded from below on it.
    By Ekeland's variational principle \cite{43}, there exists a sequence $(y_{j})\subset\overline{\mathcal{W}}_{l_{\epsilon}}\subset S$ such that
    \begin{equation}\label{(4.8)}
        E_{\epsilon}(y_{j})\to c_{\epsilon}:= \inf\limits_{\|u\|_{S}\leq l_{\epsilon}}E_{\epsilon}(u)
    \end{equation}
    and
    \begin{equation}\label{(4.9)}
        E_{\epsilon}^{'}(y_{j})\to 0 \quad \text{in}\ S^{*}\ \text{as}\ j\to \infty,
    \end{equation}
    In addition, we have $c_{\epsilon}<0$ by Lemma \ref{Le3.3}.

    Up to a subsequence, we may assume that $y_{j}\rightharpoonup u_{0}$ weakly in $S$.
    \eqref{(4.9)} is equivalent to
    \begin{equation}\label{(4.10)}
        \left|\langle E^{'}(y_{j}), \psi \rangle\right|\leq \rho_{j}\|\psi\|_{S},
    \end{equation}
    for every $\psi\in S$, where $\rho_{j}\to 0$ as $j\to \infty$.
    Taking $\psi=y_{j}-u_{0}$ in \eqref{(4.10)}, we obtain
    \begin{align*}
       &\int_{\mathbb{R}^{N}}
       (F^{N-1}(\nabla y_{j})F_{\xi}(\nabla y_{j})\nabla (y_{j}-u_{0})+V(x)|y_{j}|^{N-2}y_{j}(y_{j}-u_{0}))dx\\
       &\quad-\int_{\mathbb{R}^{N}}\frac{f(x,y_{j})}{F^{0}(x)^{\delta}}(y_{j}-u_{0}) dx-\epsilon\int_{\mathbb{R}^{N}}h(y_{j}-u_{0})dx=o_{j}(1),
    \end{align*}
    where $o_{j}(1)\to 0$ as $j\to \infty$.
    By the H\"{o}lder inequality, \eqref{(4.6)}, Lemmas \ref{Le2.2} and \ref{Le2.5}, we have
    \begin{equation*}
        \int_{\mathbb{R}^{N}}\frac{f(x,y_{j})}{F^{0}(x)^{\delta}}(y_{j}-u_{0})dx=o_{j}(1), \quad \epsilon\int_{\mathbb{R}^{N}}h(y_{j}-u_{0})dx=o_{j}(1).
    \end{equation*}
    Then
    \begin{equation}\label{(4.11)}
        \int_{\mathbb{R}^{N}}
       (F^{N-1}(\nabla y_{j})F_{\xi}(\nabla y_{j})\nabla (y_{j}-u_{0})+V(x)|y_{j}|^{N-2}y_{j}(y_{j}-u_{0}))dx=o_{j}(1).
    \end{equation}
    Moreover, since $y_{j}\rightharpoonup u_{0}$ weakly in $S$, thus
    \begin{equation}\label{(4.12)}
        \int_{\mathbb{R}^{N}}
       (F^{N-1}(\nabla u_{0})F_{\xi}(\nabla u_{0})\nabla (y_{j}-u_{0})+V(x)|u_{0}|^{N-2}u_{0}(y_{j}-u_{0}))dx=o_{j}(1).
    \end{equation}
    Subtracting \eqref{(4.12)} from \eqref{(4.11)} and using inequalities \eqref{(2.5)} and \eqref{(2.6)}, we derive that
    \begin{align*}
      &\int_{\mathbb{R}^{N}}
       \left((F^{N-1}(\nabla y_{j})F_{\xi}(\nabla y_{j})-F^{N-1}(\nabla u_{0})F_{\xi}(\nabla u_{0}))(\nabla y_{j}-\nabla u_{0})\right.\\
       &\quad\left.+V(x)(|y_{j}|^{N-2}y_{j}-|u_{0}|^{N-2}u_{0})(y_{j}-u_{0})\right)dx\\
       &\geq \int_{\mathbb{R}^{N}}\left(d_{0}F^{N}(\nabla y_{j}-\nabla u_{0})+V(x)2^{2-N}|y_{j}-u_{0}|^{N}\right)dx\\
       &\geq d_{1}\int_{\mathbb{R}^{N}}\left(F^{N}(\nabla (y_{j}-u_{0}))+V(x)|y_{j}-u_{0}|^{N}\right)dx\\
       &=d_{1}\|y_{j}-u_{0}\|_{S}^{N},
    \end{align*}
    where $d_{1}=\min\{d_{0}, 2^{2-N}\}$.
    Hence $\|y_{j}-u_{0}\|_{S}^{N}\to 0$ as $j\to\infty$, i.e. $y_{j}\to u_{0}$ strongly in $S$ as $j\to\infty$.
    Since $E_{\epsilon}\in\mathcal{C}^{1}(S,\mathbb{R})$, we have $E_{\epsilon}(u_{0})=c_{\epsilon}$ and $E_{\epsilon}^{'}(u_{0})=0$. Therefore $u_{0}$ is a weak solution of \eqref{(1.6)}; by the negative-part argument in the introduction, $u_{0}$ is positive. Since $c_{\epsilon}<0$, $u_{0}$ is a positive weak solution of \eqref{(1.6)} with negative energy.

    \textbf{Step 3: The two solutions are distinct.}

    It remains to show that, for $\epsilon>0$ sufficiently small, $u_{K}\not\equiv u_{0}$.
    We argue by contradiction. Assume that $u_{K}\equiv u_{0}$; then $w_{j}\rightharpoonup u_{0}$ weakly in $S$. Therefore

    \begin{equation}\label{(4.13)}
        E_{\epsilon}(w_{j})\to c_{K}>0, \quad \left|\langle E^{'}_{\epsilon}(w_{j}), \psi \rangle\right|\leq \rho_{j}\|\psi\|_{S},\ \forall\psi\in S,
    \end{equation}
    where $\rho_{j}\to 0$ as $j\to \infty$.
    By Lemma \ref{Le3.6}, we obtain
    \begin{equation}\label{(4.14)}
        \int_{\mathbb{R}^{N}}\frac{G(x,w_{j})}{F^{0}(x)^{\delta}}dx\to \int_{\mathbb{R}^{N}}\frac{G(x,u_{0})}{F^{0}(x)^{\delta}}dx \quad \text{as}\ j\to\infty.
    \end{equation}
   Moreover, since $w_{j}\rightharpoonup u_{0}$ weakly in $S$, by the H\"{o}lder inequality and Lemma \ref{Le2.5} we derive that
    \begin{equation}\label{(4.15)}
        \int_{\mathbb{R}^{N}}hw_{j}dx\to \int_{\mathbb{R}^{N}}hu_{0}dx \quad \text{as}\ j\to\infty.
    \end{equation}
     Combining \eqref{(4.13)}, \eqref{(4.14)} and \eqref{(4.15)}, we have
    \begin{equation}\label{(4.16)}
        \frac{1}{N}\|w_{j}\|_{S}^{N}=c_{K}+\int_{\mathbb{R}^{N}}\frac{G(x,u_{0})}{F^{0}(x)^{\delta}}dx+\epsilon\int_{\mathbb{R}^{N}}hu_{0}dx+o_{j}(1).
    \end{equation}
    Similarly, we can obtain
    \begin{equation}\label{(4.17)}
        \frac{1}{N}\|y_{j}\|_{S}^{N}=c_{\epsilon}+\int_{\mathbb{R}^{N}}\frac{G(x,u_{0})}{F^{0}(x)^{\delta}}dx+\epsilon\int_{\mathbb{R}^{N}}hu_{0}dx+o_{j}(1).
    \end{equation}
   By \eqref{(4.16)} and \eqref{(4.17)}, we have
    \begin{equation}\label{(4.18)}
        \|w_{j}\|_{S}^{N}-\|u_{0}\|_{S}^{N}=N(c_{K}-c_{\epsilon}+o_{j}(1)).
    \end{equation}
    Using Lemma \ref{Le3.5}, there exist $\epsilon^{0}>0$ and $\gamma^{0}>0$ such that, if $0<\epsilon<\epsilon^{0}$, we may take $c_{K}$ such that
    \begin{equation}\label{(4.19)}
        0<c_{K}<\frac{1}{N}\left(\frac{N-\delta}{N}\frac{\lambda_{N}}{\lambda_{0}}\right)^{N-1}-\gamma^{0}.
    \end{equation}
     Since $l_{\epsilon}\to 0$ as $\epsilon\to 0$, recalling \eqref{(4.7)} and applying the H\"{o}lder inequality and Lemma \ref{Le2.5}, we obtain
    \begin{equation*}
        \sup\limits_{\|u\|_{S}\leq l_{\epsilon}}\int_{\mathbb{R}^{N}}\frac{G(x,u)}{F^{0}(x)^{\delta}}dx\to 0
    \end{equation*}
    and
    \begin{equation*}
        \sup\limits_{\|u\|_{S}\leq l_{\epsilon}}\int_{\mathbb{R}^{N}}hudx\to 0
    \end{equation*}
    as $\epsilon\to 0$.
    Thus $c_{\epsilon}\to 0$ as $\epsilon\to 0$.
    Let $\epsilon^{*}=\min\{\epsilon^{0}, \epsilon_{2}\}$. Then there exists $\epsilon_{3}$ with $0<\epsilon_{3}<\epsilon^{*}$ such that, for $0<\epsilon<\epsilon_{3}$,
    \begin{equation}\label{(4.20)}
       0<c_{K}-c_{\epsilon}<\frac{1}{N}\left(\frac{N-\delta}{N}\frac{\lambda_{N}}{\lambda_{0}}\right)^{N-1}.
     \end{equation}
     Set
     \begin{equation*}
        m_{j}=\frac{w_{j}}{\|w_{j}\|_{S}}, \quad m_{0}=\frac{u_{0}}{(\|u_{0}\|_{S}^{N}+N(c_{K}-c_{\epsilon}))^{\frac{1}{N}}}.
     \end{equation*}
     By \eqref{(4.18)} and the weak convergence $w_{j}\rightharpoonup u_{0}$ in $S$, we obtain that $m_{j}\rightharpoonup m_{0}$ weakly in $S$.
     Note that
    \begin{equation*}
       \int_{\mathbb{R}^{N}}\frac{\Phi(N,
        \lambda_{0}|w_{j}|^{\frac{N}{N-1}})}{F^{0}(x)^{\delta}}dx=\int_{\mathbb{R}^{N}}\frac{\Phi(N,\lambda_{0}\|w_{j}\|_{S}^{\frac{N}{N-1}}|m_{j}|^{\frac{N}{N-1}})}{F^{0}(x)^{\delta}}dx.
    \end{equation*}
    Combining \eqref{(4.18)} and \eqref{(4.20)}, we have
    \begin{equation*}
     \lim\limits_{j\to\infty}
     \lambda_{0}\|w_{j}\|_{S}^{\frac{N}{N-1}}
     (1-\|m_{0}\|_{S}^{N})^{\frac{1}{N-1}}<\left(1-\frac{\delta}{N}\right)\lambda_{N}.
    \end{equation*}
    By Lemmas \ref{Le2.3} and \ref{Le2.4}, $\frac{\Phi(N, \lambda_{0}|w_{j}|^{\frac{N}{N-1}})}{F^{0}(x)^{\delta}}$ is bounded in $L^{q}(\mathbb{R}^{N})$ for some $q>1$. By $(B_{1})$, we obtain
    \begin{equation*}
        |f(x, w_{j})|\leq a_{1}|w_{j}|^{N-1}+a_{2}\Phi\left(N,\lambda_{0}|w_{j}|^{\frac{N}{N-1}}\right).
    \end{equation*}
    Then, by the continuous embedding $S\hookrightarrow L^{p}(\mathbb{R}^{N})$ for every $p\geq 1$, we derive that $\frac{f(x, w_{j})}{F^{0}(x)^{\delta}}$ is bounded in $L^{q_{1}}(\mathbb{R}^{N})$ for some $q_{1}$ with $1<q_{1}<q$. This, together with Lemma \ref{Le2.5} and the H\"{o}lder inequality, implies
    \begin{equation}\label{(4.21)}
        \left|\int_{\mathbb{R}^{N}}\frac{f(x, w_{j})(w_{j}-u_{0})}{F^{0}(x)^{\delta}}dx\right|\leq \left\|\frac{f(x,w_{j})}{F^{0}(x)^{\delta}}\right\|_{L^{q_{1}}(\mathbb{R}^{N})} \|w_{j}-u_{0}\|_{L^{q_{2}}(\mathbb{R}^{N})}\to 0,
    \end{equation}
    where $\frac{1}{q_{1}}+\frac{1}{q_{2}}=1$.

    Taking $\psi=w_{j}-u_{0}$ in \eqref{(4.13)}, by \eqref{(4.15)} and \eqref{(4.21)}, we have
     \begin{equation}\label{(4.22)}
        \int_{\mathbb{R}^{N}}
       (F^{N-1}(\nabla w_{j})F_{\xi}(\nabla w_{j})\nabla (w_{j}-u_{0})+V(x)|w_{j}|^{N-2}w_{j}(w_{j}-u_{0}))dx=o_{j}(1).
     \end{equation}
    Then by the weak convergence $w_{j}\rightharpoonup u_{0}$ in $S$,  we obtain
     \begin{equation}\label{(4.23)}
        \int_{\mathbb{R}^{N}}
       (F^{N-1}(\nabla u_{0})F_{\xi}(\nabla u_{0})\nabla (w_{j}-u_{0})+V(x)|u_{0}|^{N-2}u_{0}(w_{j}-u_{0}))dx=o_{j}(1).
    \end{equation}
    Subtracting \eqref{(4.23)} from \eqref{(4.22)} and applying inequalities \eqref{(2.5)} and \eqref{(2.6)}, it follows that
    \begin{equation*}
        \|w_{j}-u_{0}\|_{S}^{N}\to 0 \quad \text{as}\ j\to\infty.
    \end{equation*}
    Combining with \eqref{(4.18)}, we have
    \begin{equation*}
        c_{K}=c_{\epsilon},
    \end{equation*}
    which contradicts $c_{K}>0$ and $c_{\epsilon}<0$, thus $u_{K}\not\equiv u_{0}$. Consequently, for every $0<\epsilon<\epsilon_{3}$, problem \eqref{(1.6)} admits two distinct positive weak solutions, namely $u_{K}$ and $u_{0}$. The proof of Theorem \ref{Th1.2} is completed. \qed

\section*{Acknowledgments}Supported by Natural Science Foundation of China (Grant Nos. 11721101, 11971026, 12171459) and Natural Science Foundation of Education Committee of Anhui Province (Grant No. 2024AH051344).

      \end{document}